\documentclass[reqno]{amsart}
\usepackage[utf8x]{inputenc}
\usepackage{amsmath, amsthm, amssymb}
\usepackage[usenames,dvipsnames,svgnames,table]{xcolor}
\usepackage[margin=1.5in]{geometry}

\newtheorem{theorem}{Theorem}[section]
\newtheorem{proposition}[theorem]{Proposition}
\newtheorem{lemma}[theorem]{Lemma}

\newcommand{\R}{\mathbb R}

\newcommand{\eps}{\varepsilon}

\newcommand{\dd}{\, \mathrm{d}}

\newcommand{\mb}{\mathbf}

\DeclareMathOperator{\sech}{sech}
\DeclareMathOperator{\sgn}{sgn}

\numberwithin{equation}{section}

\makeatletter
\@namedef{subjclassname@2020}{%
  \textup{2020} Mathematics Subject Classification}
\makeatother

\title[Near-constant solutions]{Existence and stability of near-constant solutions 
of scalar field equations}

\author{Mashael Alammari}
\email{malammari2012@my.fit.edu}

\author{Stanley Snelson}
\email{ssnelson@fit.edu}

\thanks{SS was partially supported by a Collaboration Grant from the Simons Foundation, Award \# 855061}

\begin{document}

\maketitle

\begin{center}
Department of Mathematical Sciences, Florida Institute of Technology, Melbourne, FL 
\end{center}

\begin{abstract}
This article studies a class of semilinear scalar field equations on the real line with variable coefficients in the linear terms. These coefficients are not necessarily small perturbations of a constant. We prove that under suitable conditions, the non-translation-invariant linear operator leads to steady states that are ``almost constant'' in the spatial variable. The main challenge of the proof is due to a spectral obstruction that cannot be treated perturbatively. Next, we consider stability of constant and near-constant steady states. We establish asymptotic stability for the vacuum state with respect to perturbations in $H^1\times L^2$, without placing any parity assumptions on the coefficients, potential, or initial data. Finally, under a parity assumption, we show asymptotic stability for near-constant steady states.
\end{abstract}

\section{Introduction}

We are concerned with the long-time behavior of solutions to semilinear scalar-field equations on the real line of the form
\begin{equation}\label{e:main}
\partial_t^2 u - Lu + F'(u) = 0, \quad (t,x) \in [0,\infty)\times \R,
\end{equation}
where $L$ is a linear second-order operator in the spatial variable:
\[ L u := a(x) \partial_x^2 u + b(x) \partial_x u + c(x) u.\]
Scalar-field models are used to investigate a wide variety of physical phenomena, such as fundamental particles in quantum field theory, cosmological inflation, and phase transitions in condensed matter physics; see e.g. \cite{kinks-domainwalls, lohe,khare,top-sol} for more background on scalar field equations.  In particular, our assumptions include (variable-coefficient versions of) the $\phi^4$ model (see \cite{segur1987, cuccagna2008kink, KMMphi4, ross2019}), the sine-Gordon equation \cite{ivancevic-sinegordon, cuenda2011sg, sine-gordon-book, chen2020sine-gordon, AMP2020sinegordon}, and higher-order field theories such as the $\phi^6$ model \cite{lohe}.

The stability of special solutions for scalar field equations is widely studied, but usually in the case of constant coefficients (most commonly $a(x) \equiv 1$, $b(x) \equiv 0$, $c(x) \equiv 0$). We would like to understand how the variable coefficients $a(x)$, $b(x)$, $c(x)$ in $L$ affect the stability theory of special solutions. Since the linear operator $L$ typically encodes the medium through which a disturbance propagates, handling a non-constant medium is desirable on physical grounds. We will also  see that the mathematical analysis of steady states (in terms of both existence and stability) is nontrivially affected by the presence of variable coefficients.

Our precise assumptions will be given below, but let us emphasize that our $L$ is not necessarily a perturbation of a given constant-coefficient operator.

%

We are interested in two types of stationary solutions of \eqref{e:main}, corresponding to $\xi \in \R$ such that $F'(\xi) = 0$:
\begin{enumerate}
\item Constant solutions $U(x)\equiv \xi$, in the cases $c(x) \equiv 0$ or $\xi = 0$. 
\item Non-constant solutions $U(x)$ decaying to $\xi$ as $x\to\pm\infty$, that are uniformly close to $\xi$. We construct these near-constant steady states, under suitable conditions on the coefficients and potential, in Theorem \ref{t:near-constant}. 
\end{enumerate}

\subsection{The class of equations considered}
Following \cite{snelson2016stability, alammari2021} we change variables to remove the second-order coefficient $a(x)$: let $y = \int_0^x a^{-1/2}(z) \dd z$ and write $u(t,y) =  u(t,x(y))$. Changing the definition of $b$, we write $b(y) = b(x(y)) - a^{-1/2}(x(y))\frac d {dy} a^{1/2}(x(y))$. 
This leads to the equation
\begin{equation}\label{e:main-y}
\partial_t^2 u - \partial_y^2 u -b(y)\partial_y u - c(y) u= -F'(u), \quad (t,y)\in \R\times \R.
\end{equation}
Our precise assumptions are stated in terms of the $y$ variable. We always assume 
\[F \in C^4(\R), \quad b, c \in L^1(\R) \cap L^\infty(\R).\]
For certain results, we will place extra conditions on $b$, $c$, and $F$.

We note that the energy 
\begin{equation}\label{e:energy}
 E(u(t,\cdot)) = \int_{\R} \left[ \frac 1 2 (\partial_t u)^2 + \frac 1 2 (\partial_y u)^2 - \frac 1 2 c(y)u^2 + F(u)\right] \omega(y)\dd y,
 \end{equation}
is conserved by the flow of \eqref{e:main-y}, where $\omega(y) = \exp(\int_{-\infty}^y b(z) \dd z)$. In order to ensure that the constant state $u \equiv \xi$ has finite energy, we assume (by shifting the potential $F$ if necessary, which does not affect the equation of motion) that
\[ F(\xi) = F'(\xi) = 0.\]


\subsection{Main results}

First, we address the existence of stationary solutions. If the zeroth-order coefficient $c(y)$ is identically $0$ and $F'(\xi) = 0$, then $u(y)\equiv \xi$ is obviously a stationary solution to \eqref{e:main-y}. On the other hand, if $c$ is nonzero and $\xi\neq 0$, then $u\equiv \xi$ is no longer a solution, but we can prove the existence of a stationary state $U(y)$ that is uniformly close to $\xi$ and decays to $\xi$ at $\pm \infty$:

%
%
%
%
%
%


\begin{theorem}\label{t:near-constant}
With $b$, $c$, and $F$ as above, assume in addition that $F''(\xi)>0$, and that
\begin{equation}\label{e:smallness}
 \|c\|_{L^1(\R)} + \|c\|_{L^\infty(\R)} \leq \min\left\{1, \frac\delta {|\xi|}\right\},
\end{equation}
for a sufficiently small $\delta>0$ depending on $b$ and $F$. We take $\min\left\{1,\delta/|\xi|\right\} = 1$ if $\xi = 0$. 

\smallskip

%
%
%
%
%
%
%
%
%
%
%

\noindent {\bf Case 1:} If $0$ is {\bf not} an $L^2(\R)$-eigenvalue of the linear operator
\[ \mathcal L_{b,c} := -\partial_y^2 - b\partial_y -(c-F''(\xi)),\]
then 
%
there exists a stationary solution $U(y)$ to \eqref{e:main-y}, with 
\begin{equation}\label{e:UxiL1}
 \|U-\xi\|_{L^1(\R)} + \|U-\xi\|_{L^\infty(\R)} +\|U'\|_{L^1(\R)} + \|U'\|_{L^\infty(\R)} \lesssim \delta.
 \end{equation}
If in addition, $c(y)$ satisfies $|c(y)| \leq K e^{-k|y|}$ for some $K \leq \min\{1, \delta/|\xi|\}$ and some $0< k < \sqrt{F''(\xi)}$, 
then we have the improved decay estimates
\[ \|e^{k|y|}(U(y)-\xi)\|_{L^\infty(\R)} + \|e^{k|y|}U'(y)\|_{L^\infty(\R)} \lesssim \delta.\]

\smallskip

\noindent {\bf Case 2:} If $0$ {\bf is} an $L^2(\R)$-eigenvalue of $\mathcal L_{b,c}$, assume in addition that 
$F'''(\xi) = 0$, and $F^{(4)}(\xi) \neq 0$. Then 
there exists a stationary solution $U(y)$ to \eqref{e:main-y}, with
\[ \|U-\xi\|_{L^1(\R)} + \|U-\xi\|_{L^\infty(\R)} +\|U'\|_{L^1(\R)} + \|U'\|_{L^\infty(\R)} \lesssim \delta^{1/3}.\]
As in the previous case, if $c$ satisfies $|c(y)| \leq K e^{-k|y|}$ with $0< k< \sqrt{F''(\xi)}$, then we have
\[ \|e^{k|y|}(U(y)-\xi)\|_{L^\infty(\R)} + \|e^{k|y|}U'(y)\|_{L^\infty(\R)} \lesssim \delta^{1/3}.\]

In either Case 1 or Case 2, if $b$ is odd, $c$ is even, and $F(\xi + \cdot)$ is even, then $U$ is odd.
\end{theorem}


Note that $F(u) = 1-\cos(u)$ with $\xi = k\pi$, $k\neq 0$, is admissible in Case 2. 

Let us comment on the proof of Theorem \ref{t:near-constant}. Writing $U = \xi + U_\delta$, one obtains an equation of the form $\mathcal L_{b,c} U_\delta = c\xi+\mathcal N(U_\delta)$, where $\mathcal L_{b,c}$ is the linearization around $\xi$, and $\mathcal N(U_\delta)$ is a nonlinearity that depends on $F$. In Case 1, following a classical strategy, we convert the equation for $U_\delta$ into an integral equation via the Green's function for $\mathcal L_{b,c}$ and solve this integral equation via contraction. Case 2 is more difficult, because $\mathcal L_{b,c}$ is only invertible on the orthogonal complement of the $0$-eigenspace $E_0$. To get around this obstacle, we decompose $U_\delta$ into its projections onto $E_0$ and $E_0^\perp$, and the main difficulty of the proof becomes controlling the nonlinear interactions between these two components, arising from the term $\mathcal N(U_\delta)$. One key observation is that these two components of $U_\delta$ are of different orders in the parameter $\delta$. In particular, the $E_0$ component is the reason for the coarser upper bound $\|U-\xi\|\lesssim \delta^{1/3}$ in Case 2.

As mentioned above, we do not make any smallness assumption on the coefficient $b$. Therefore, even though the steady state $U$ is a perturbation of a constant, we cannot treat the equation itself as a perturbation of some constant-coefficient equation whose properties are well understood.

Next, we address the stability of constant and near-constant steady states $U$. We study perturbations of $U$ in the \emph{energy space} $H^1(\R)\times L^2(\R)$. By standard techniques, \eqref{e:main} is locally well-posed for initial data $(u,\partial_t u)|_{t=0} = (U+v_1,v_2)$, with $(v_1,v_2)\in H^1(\R)\times L^2(\R)$. 
%
After showing orbital stability in Proposition \ref{p:orbital-constant}, we prove asymptotic stability in two main regimes. First, we address the vacuum solution, with no parity condition on $F$, $b$, $c$, or the initial data:
\begin{theorem}\label{t:asymptotic}
Let $F(0) = F'(0) = 0$, $F''(0)>0$, and 
assume that $b, c \in W^{1,\infty}(\R)$ satisfy the conditions
\begin{equation}\label{e:b-c-coercive}
4\lambda \tanh(y/\lambda) b(y) \leq \sech^2(y/\lambda), \quad 2\lambda \tanh(y/\lambda) c' + (\sech^2(y/\lambda) b)' \geq 8\sech^2(y/\lambda),
 \end{equation}
for some $\lambda>2$. 

Then there exists $\eps>0$ such that if the initial data $(u_0,u_1)$ satisfies
\[ \|u_0 \|_{H^1(\R)} + \|u_1\|_{L^2} < \eps,\]
then the vacuum solution is asymptotically stable in the sense that
 \[ \lim_{t\to \infty} \left(\|u(t)\|_{H^1(I)} + \|\partial_t u(t)\|_{L^2(I)}\right) = 0,\]
for any bounded interval $I\subset \R$.
\end{theorem}

Let us make the following remarks on Theorem \ref{t:asymptotic}:
\begin{itemize}
\item The conclusion cannot be improved by replacing $I$ with $\R$, because such a conclusion would violate energy conservation. 



\item Condition \eqref{e:b-c-coercive} arises from the proof. It is a condition both on the growth rate of $b$ and $c$ for small values of $y$, and their exponential decay. Because of this, it is difficult to replace with a more concise condition. The set of $b, c$ satisfying \eqref{e:b-c-coercive} is not empty. To take a crude example, for $\lambda > 12$ let $c(y) = -8\lambda^4 \sech(2y/\lambda)$, and
\[ b(y) = \begin{cases} 16y, &|y|< 1/\lambda,\\
\sgn(y) (16/\lambda)e^{-(10/\lambda)(|y|-1/\lambda)}, &|y|> 1/\lambda.\end{cases}\]

\item Theorem \ref{t:asymptotic} applies to any sufficiently smooth potential $F$ with a stable potential well at $0$. This lies in contrast to the constant-coefficient case, where stability of vacuum seems to depend sensitively on the choice of $F$. For example, the vacuum state in the constant-speed sine-Gordon equation ($F'(u) = \sin(u)$) is not asymptotically stable due to the presence of breathers \cite{cuenda2011sg}, but in the scalar field model with $F'(u) = u - u^3$ (which differs from $\sin(u)$ only at the third order), the vacuum state is conjectured to be asymptotically stable \cite{soffer1999resonance}, and this has been confirmed for the special case of odd perturbations \cite{KMM2017breathers}.

\item In the original $x$ variables used to state our equation \eqref{e:main}, the hypotheses of Theorem \ref{t:asymptotic} would read as follows: $a(x) \in W^{2,\infty}(\R)$ with $a(x) \geq a_0>0$ for all $x$, and $b(x), c(x) \in W^{1,\infty}(\R)$, satisfying, with $y(x) = \int_0^x a^{-1/2}(z) \dd z$,
\[
\begin{split}
&4 \lambda \tanh(y(x)/\lambda)\left(b(x) - \frac{da/dx}{\sqrt{a(x)}}\right) \leq \sech^2(y(x)/\lambda),\\
&\sqrt{a(x)}\left[2\lambda \tanh(y(x)/\lambda)  \frac {dc}{dx} + \frac d {dx}\left(\sech^2(y(x)/\lambda) \left(b(x) - \frac{a'(x)}{\sqrt{a(x)}}\right)\right)\right] \geq 8\sech^2(y(x)/\lambda).  
\end{split}
\]

\end{itemize}

Next, we address the stability of near-constant states $U$ with respect to odd perturbations. Here, we take $\xi \neq 0$ because the near-constant solutions decaying to $0$ are in fact identically zero (as can be seen from the proof of Theorem \ref{t:near-constant}), so the case $\xi = 0$ is already covered by Theorem \ref{t:asymptotic}. We focus on the sine-Gordon equation because the nonlinearity around $\xi = 2k\pi$, $k\neq 0$, is odd. However, the method of proof can be generalized to other potentials by deriving a corresponding version of Lemma \ref{l:Fterm}.

\begin{theorem}\label{t:asymptotic-odd}
Let $\xi = 2k\pi$, $k\neq 0$, let $F(u) = 1-\cos(u)$ in \eqref{e:main-y}, and let $U(y)$ be the near-constant state guaranteed by Theorem \ref{t:near-constant}. Assume there exists $\mu>0$ such that 
\begin{equation}\label{e:cyF}
c(y)\leq 1 -\mu, \quad y\in \R,
\end{equation}
(so that, by Proposition \ref{p:orbital-constant}, $U(y)$ is orbitally stable). Assume in addition that 
$ |c(y)| \lesssim  e^{-|y|}$, 
and that $b, c \in W^{1,\infty}(\R)$ satisfy the sign conditions
\begin{equation}\label{e:sign-condition}
 \sgn(y)b(y)\leq 0, \quad \sgn(y)b'(y)\geq 0, \quad \sgn(y)c'(y) \geq 0.
 \end{equation}

Then there exists $\eps>0$ such that if the initial data $(u_0,u_1) = (U+v_1,v_2)$ with $(v_1,v_2)$ odd in $y$, and satisfies
\[ \|u_0 - U \|_{H^1(\R)} + \|u_1\|_{L^2} < \eps,\]
then
 \[ \lim_{t\to \infty} \left(\|u(t)-U\|_{H^1(I)} + \|\partial_t u(t)\|_{L^2(I)}\right) = 0,\]
for any bounded interval $I\subset \R$.
\end{theorem}

This theorem should be contrasted with the behavior seen in the constant speed sine-Gordon equation $\partial_t^2 u - \partial_x^2 u = -\sin u$. Since this constant-coefficient equation is invariant under $u(x)\mapsto u(x)+2k\pi$, the constant solutions $u(x) \equiv 2k\pi$ have the same stability properties as the vacuum state $u(x) \equiv 0$. It is well-known that this vacuum state is not asymptotically stable in the energy space, due to the existence of breather solutions that oscillate near zero (or $2k\pi$) and do not decay as $t\to \infty$ (see, e.g. \cite{cuenda2011sg}). Therefore, our result shows that sine-Gordon breathers disappear under perturbations of the linear part of the equation.

The proof strategy for Theorems \ref{t:asymptotic} and \ref{t:asymptotic-odd}, adapted from \cite{KMMphi4, KMM2017breathers} is based on a Virial functional of the form
\[\mathcal I(v) = \int_\R \left(\lambda\tanh(y/\lambda) \partial_y v + \frac 1 2 \tanh'(y/\lambda) v\right) \partial_t v \dd y,\]
for some $\lambda >0$, where $v$ is the difference between the solution $u$ and the steady-state $U$. 
The functional $\mathcal I$ is a Lyaponov-type functional satisfying (a) coercivity of $\frac d {dt}\mathcal I$ with respect to the (weighted) $H^1$ norm of $v$, and (b) $\mathcal I$ itself is controlled from above by $H^1\times L^2$ norms of $v$.
 Establishing property (a) is the main ingredient of the proofs. 

This type of Virial argument was pioneered by Kowalcyz-Martel-Mu\~noz, who used Virial functionals of the same form to prove asymptotic stability with respect to odd perturbations for the $\phi^4$ kink \cite{KMMphi4}, and also the absense of odd breathers close to the vacuum solution for a class of equations with potential $F$ sufficiently smooth near zero \cite{KMM2017breathers}. Related techniques were later used in \cite{KMMV2020kink,KMkink, cuccagna2022kink, li2022quadratic} to establish asymptotic stability results for various classes of scalar field models. See also \cite{cuccagna2019nls, cuccagna2021nls} for related approaches in the context of nonlinear Schr\"odinger equations with trapping potential. 

These works all applied to equations with constant-coefficient linear terms. In \cite{snelson2016stability}, the second named author generalized the result of \cite{KMMphi4} to the case of variable coefficients, but with $c=0$ and $b$ small. The present article deals with non-perturbative coefficients $b$ and $c$. 
One novelty of our proof compared to \cite{snelson2016stability} is  
in extending the coercivity property of $\frac d{dt} \mathcal I$ to functions that are not necessarily odd. The mechanism for this coercivity is provided by the linear coefficients $b$ and $c$, which in some sense act as damping terms. 

If $b$ and $c$ are both small in a suitable norm, analysis along the lines of the proof of Theorem \ref{t:asymptotic-odd} can establish asymptotic stability of the vacuum state with respect to odd perturbations, without the assumption \eqref{e:b-c-coercive}. 
We omit the details, since the key ideas are similar to those contained in this article, but simpler.

\subsection{Comparison with constant-coefficient case}\label{s:compare}


In the constant-coefficient regime, asymptotic stability of constant steady-states does not hold in general, as shown by the example mentioned above of the sine-Gordon equation with $F(u) = 1-\cos(u)$. 
Breathers in constant-speed sine-Gordon are even in $x$, and it is known that there are no \emph{odd} breathers close to the vacuum solution, by the result of \cite{KMM2017breathers}. Sine-Gordon breathers are related to the integrability of the system, and in general breathers are understood to be a rare phenomenon. 

For the variable-coefficient equation \eqref{e:main-y}, we identify in this article a class of pairs $(L,F)$ such that constant and near-constant solutions are asymptoticaly stable. The linear terms in $L$ turn out to be the key factor in deriving this stability result, which is somehow less sensitive to the choice of $F$ than in the constant-speed case. For example, we establish asymptotic stability for certain variable-speed sine-Gordon equations, even though asymptotic stability of vacuum does not hold for the constant-speed equation, as noted in the previous paragraph.


\subsection{Related work} Stability of vacuum has been studied intensively for nonlinear dispersive equations, and a full bibliography is outside the scope of this introduction. We refer the reader to results on 
 nonlinear Klein-Gordon \cite{bambusi2011, sterbenz2016decay, germain2020nlkg, delort2001nlkg, lindblad2015scattering, lindblad2019scattering}, nonlinear Schr\"odinger \cite{hayashi1998nls, ifrim2015nls}, and other dispersive equations \cite{ifrim2019b-o,hgm2018cc,martinez2020hartree, lindblad-tao}. Regarding variable coefficients, we refer to, e.g. \cite{cuccagna2014nls, delort2016nls, germain2018nls-potential, naumkin2016nls} for recent results on the long-time decay of small solutions for NLS models with variable zeroth-order linear terms. 
More generally, there are ``no-breathers'' results that rule out global, uniformly small solutions that do not vanish as $t\to \infty$: see \cite{segur1987, vuillermot, denzler1993sinegordon, KMM2017breathers} and the references therein. See also the recent work \cite{kohler2021} which established existence of breather solutions for a class of variable-coefficient wave equations on $\R\times\R$, which are quasilinear and therefore do not overlap with our setting.





\section{Existence of near-constant stationary solutions}\label{s:exist}

This section is devoted to the proof of Theorem \ref{t:near-constant}. To begin, we write $U = \xi + U_\delta$ and plug this ansatz into $-U''-bU'-cU = -F'(U)$ to obtain
\begin{equation}\label{e:Udelta}
 -U_\delta'' - bU_\delta' - (c-F''(\xi)) U_\delta = c\xi + \mathcal N(U_\delta),
 \end{equation}
where $\mathcal N(U_\delta) = -F'(\xi+U_\delta) + F''(\xi)U_\delta$. Define $\mathcal L_{b,c} = -\partial_y^2 - b \partial_y - (c-F''(\xi))$. This operator encodes the linear part of \eqref{e:Udelta}. We would like to find solutions $Y_{\pm\infty}$ to the linear equation $\mathcal L_{b,c} Y = 0$, 
which can be written in system form, with $\mb Y = ( Y,Y')$, as
\begin{equation}\label{e:Y-system}
 \mb Y' = \left(\left(\begin{array}{cc}0 & 1\\ F''(\xi) & 0     \end{array}\right) +  \left(\begin{array}{cc}0 & 0\\ -c(y)   & -b(y)     \end{array}\right)\right)\mb Y. 
 \end{equation}
Let $m = \sqrt{F''(\xi)}>0$. Using ODE techniques (for a detailed proof, see \cite[Lemma A.2]{alammari2021}), there exist solutions $\mb Y_{\pm\infty}$ to this system, defined on all of $\R$, with 
\[\lim_{y\to \pm\infty} e^{\pm m y}\mb Y_{\pm \infty}(y) = \left(\begin{array}{c} 1\\\mp m\end{array}\right).\]
We also have $|Y_{\pm\infty}(y)|\leq C e^{\mp m|y|}$ globally on $\R$, for a constant depending on $F$, $b$, and $c$. However, the dependence on $c$ is through an upper bound for $\|c\|_{L^1(\R)}$, which is bounded by 1 as a result of \eqref{e:smallness}. 

 Let $W_{\mb Y}(y) = \det(\mb Y_{-\infty} ,\mb Y_{\infty})$ be the Wronskian of $\mb Y_\infty$ and $\mb Y_{-\infty}$. Abel's formula implies $W_{\mb Y}(y) = W_{\mb Y}(0)\exp(\int_{0}^y b(z) \dd z)$. Since $b\in L^1(\R)$, this exponential is bounded away from zero. As a result, $\mb Y_\infty(y)$ and $\mb Y_{-\infty}(y)$ are either parallel for all $y\in \R$ (in which case $0$ is an $L^2$ eigenvalue of $\mathcal L_{b,c}$), or linearly independent for all $y\in \R$. We address these two cases in turn.

\subsection{Case 1: $0$ is not an eigenvalue of $\mathcal L_{b,c}$} 

The techniques for this case (which is the simpler case) are largely inspired by our earlier work \cite[Theorem 1.1]{alammari2021}. We give the full details for the convenience of the reader. 

In this case, $\mb Y_{\pm\infty}$ are linearly independent, and we can  define the Green's function
\[ G(y,w) := \frac 1 {W_{\mb Y}(w)} \begin{cases} Y_{-\infty}(y) Y_\infty(w), & y< w,\\
Y_{\infty}(y)Y_{-\infty}(w), & w\leq y,\end{cases}\]
The following estimate for the Green's function can be found, for example, in the proof of \cite[Theorem 1.1]{alammari2021}: with $\|\cdot\|_{X} = \|\cdot\|_{L^1(\R)} + \|\cdot \|_{L^\infty(\R)}$,
\begin{equation}\label{e:X-norm}
 \left\|\int_\R G(\cdot, w) \eta(w) \dd w \right\|_{X} + \left\|\partial_y \int_\R G(\cdot, w) \eta(w) \dd w \right\|_{X}  \leq C \|\eta\|_{X},
 \end{equation}
for a constant $C$ depending on $\|b\|_{L^1(\R)}$, $\|c\|_{L^1(\R)}$ (which is bounded by 1 from \eqref{e:smallness}), and $m$. 

This Green's function $G$ allows us to write \eqref{e:Udelta} as 
\begin{equation}\label{e:Udelta-integral}
 U_\delta(y) = (\mathcal T U_\delta)(y) := \xi \int_\R G(y,w) c(w) \dd w + \int_\R G(y,w) \mathcal N(U_\delta)(w) \dd w.
 \end{equation}
Our goal is to solve this integral equation via a fixed point argument. Regarding the first term in \eqref{e:Udelta-integral}, we have
\[ \left\|\xi \int_\R G(\cdot,w) c(w) \dd w\right\|_X \leq C|\xi| \|c(w)\|_{X}.\]
As a result of \eqref{e:smallness}, this term is bounded by a constant $C_0$ times $\delta$. 

For the second term in \eqref{e:Udelta-integral}, since $F$ is $C^3$, we have, for $\eta\in X$,
\begin{equation}\label{e:nonlinearity}
  |\mathcal N(\eta)| =\left|-F'(\xi+\eta) + F''(\xi)\eta \right| \leq K \eta^2,    
  \end{equation}
for some $K>0$ depending on the $C^3$ norm of $F$. Therefore, 
\[\begin{split}
 \|\mathcal T \eta\|_{X} &\leq C_0\delta + K \left\|\int_\R G(\cdot,w) \eta^2 (w) \dd w\right\|_X \\
 &\leq C_0\delta + K \|\eta\|_X^2, 
 \end{split}\]
 where the value of $K$ changes line-by-line, and we have used the obvious interpolation $\|\eta\|_{L^2} \leq \|\eta\|_{L^1}^{1/2}\|\eta\|_{L^\infty}^{1/2}$ in the last inequality.
 
 Defining $\mathcal A_\delta := \{\eta\in X : \|\eta\|_X \leq 2C_0 \delta\}$, we claim $\mathcal T$ maps $\mathcal A_\delta$ into itself, if $\delta$ is small enough. Indeed, our estimates imply that for $\eta\in \mathcal A_\delta$,
 \[ \|\mathcal T\eta\|_X \leq C_0 \delta + K \|\eta\|_{X}^2 \leq C_0 \delta + 4C_0^2 K \delta^2. \]
 Choosing $\delta < 1/(4C_0K)$, we see $\mathcal T\eta \in \mathcal A_\delta$, as claimed.

 Next, for $\eta_1, \eta_2 \in \mathcal A$, we have from Taylor's Theorem that
\[ F'(\xi+\eta_1) = F'(\xi+\eta_2) + F''(\xi+\eta_2)(\eta_1-\eta_2) + \frac 1 2 F'''(\zeta_y)(\eta_1-\eta_2)^2,\]
for some $\zeta_y \in \R$ depending on $y$. Using this in $\mathcal N(\eta_1) - \mathcal N(\eta_2)$, we have
\begin{equation}\label{e:difference}
\begin{split}
 |\mathcal N(\eta_1) - \mathcal N(\eta_2)| &=| F'(\xi+\eta_1) - F'(\xi+\eta_2) - F''(\xi)(\eta_1-\eta_2)|\\
 &= \left|[F''(\xi+\eta_2) - F''(\xi)](\eta_1 - \eta_2) +\frac 1 2 F'''(\zeta_y)(\eta_1 - \eta_2)^2\right|\\
 &\leq |\max |F'''(s)||\eta_2||\eta_1-\eta_2| + \frac 1 2 |F'''(\zeta_y)|(\eta_1-\eta_2)^2\\
 &\leq C \delta|\eta_1 - \eta_2|,
 \end{split}
 \end{equation}
for a constant $C>0$. Our estimates for $G$ now imply
 \[ \|\mathcal T\eta_1 - \mathcal T \eta_2\|_X \leq \left\| \int_\R G(\cdot,w) [\mathcal N(\eta_1) - N(\eta_2)] \dd w \right\|_X \leq C \delta \|\eta_1-\eta_2\|_X.\]
 For $\delta>0$ small enough, $\mathcal T$ is a contraction on $\mathcal A_\delta$, and a unique solution $U_\delta$ to \eqref{e:Udelta-integral} exists in $\mathcal A_\delta$. 
 
 To obtain the claimed estimate for the derivative $U' = U_\delta'$, we differentiate \eqref{e:Udelta-integral} and use \eqref{e:X-norm} and \eqref{e:nonlinearity}:
 \begin{equation}\label{e:deriv-est}
  \begin{split}
 \|U_\delta'\|_{X} &= \left\|\partial_y \int_\R G(y,w) [\xi c(w) + \mathcal N(U_\delta)(w)] \dd w\right\|_X \\
 & \leq C \left( |\xi| \|c(w)\|_X + \|\mathcal N(U_\delta)\|_X \right)\\
 & \leq C\delta + K \|U_\delta^2\|_X  \lesssim \delta,  \end{split}
 \end{equation}
 where, as above, we have used that $|\xi|\|c(w)\|_X \lesssim \delta$ because of assumption \eqref{e:smallness}.

Next, assume that $c(y)$ has exponential decay of order $\lesssim e^{-k|y|}$ for some $k>0$. For fixed $k$, define 
\[\|h\|_{\sim} := \sup_{y\in\R} e^{k|y|} |h(y)|.\]
We claim the following additional estimate for the operator $\eta \mapsto \int_\R G(\cdot,w)\eta(w)\dd w$:
\begin{equation}\label{e:exp-bound}
 \left\|\int_\R G(\cdot,w) \eta(w)\dd w \right\|_{\sim} \leq C \| \eta\|_{\sim}, \quad \text{ if } 0< k < m,
 \end{equation}
for some $C>0$ depending on $k$ and $m$. Indeed, since $Y_{\pm \infty}(y) \sim e^{\mp m y}$ as $y\to \pm \infty$, we have
\[ \begin{split}
\int_\R G(y,w) \eta(w) \dd w &= \frac 1 {W_{\mb Y}(y)} \left[ Y_\infty(y)\int_{-\infty}^y Y_{-\infty}(w) \eta(w) \dd w + Y_{-\infty}(y) \int_y^\infty Y_\infty(w) \eta(w) \dd w\right] \\
& \leq C \|\eta\|_{\sim} \left[ e^{-my} \int_{-\infty}^y e^{mw} e^{-k|w|} \dd w + e^{my} \int_y^\infty e^{-mw} e^{-k|w|} \dd w\right].
\end{split} \]
Focusing on the first term on the right, when $y<0$ we have
\[ e^{-my} \int_{-\infty}^y e^{mw} e^{kw} \dd w \lesssim e^{ky} = e^{-k|y|},\]
and when $y\geq 0$ we have
\[ e^{-my} \int_{-\infty}^y e^{mw} e^{-k|w|} \dd w \lesssim e^{-my}\left( \int_{-\infty}^0 e^{(m+k)w} \dd w + \int_0^y e^{(m-k)w} \dd w \right)  = e^{-k|y|},\]
since $k<m$. After applying a symmetric argument to the integral over $[y,\infty)$, we have shown \eqref{e:exp-bound}.

We now apply a contraction mapping argument in the $\|\cdot\|_{\sim}$ norm. With $\mathcal T$ defined as in \eqref{e:Udelta-integral}, the estimate \eqref{e:exp-bound} and $c(w) \lesssim e^{-k|w|}$ imply
\[ \left\|\xi \int_\R G(\cdot,w) c(w) \dd w\right\|_{\sim} \leq C |\xi| \|c(w)\|_{\sim} .\]
Because of our smallness assumption on $|\xi|\|c\|_{\sim}$, this term is bounded by $C_0\delta$ for some constant $C_0>0$. By \eqref{e:nonlinearity}, we have
\[ \|\mathcal T \eta\|_{\sim} \leq C_0 \delta + K \left\|\int_\R G(\cdot,w) \eta^2(w) \dd w\right\|_{\sim} \leq C_0 \delta + K \|\eta^2\|_{\sim} \leq C_0 \delta + K\|\eta\|_{\sim}^2,\]
and we may apply a contraction argument similar to above to conclude $\|U_\delta\|_{\sim}$ is finite, where $U_\delta$ is the solution constructed above in the space $X$.

To establish the bounds on $\partial_y U_\delta$ in the $\|\cdot\|_{\sim}$ norm, we first show
\[ \left\| \partial_y \int_\R G(y,w) \eta(w) \dd w\right\|_{\sim} \leq C\|\eta\|_{\sim},\]
following the proof of \eqref{e:exp-bound} exactly (since $Y_{\pm\infty}'(y) \lesssim e^{\mp m|y|}$). The estimate $\|U_\delta'\|_{\sim} \lesssim \delta$ now follows from differentiating \eqref{e:Udelta-integral} and proceeding as in \eqref{e:deriv-est}.

\subsection{Case 2: $0$ is an $L^2(\R)$-eigenvalue of $\mathcal L_{b,c}$}
This is the more delicate case, as we cannot proceed via a simple contraction as above because of the spectral obstruction. 

We define the weight
\[ \omega(y) = e^{-\int_0^y b(s) \dd s},\]
and inner product
 \[ \langle f,g\rangle = \int_\R \omega(y) f(y)g(y) \dd y.\]
 By an abuse of notation, we use $\langle f, g\rangle$ whenever this integral is well-defined, even if $f$ and $g$ are not both in $L^2(\R)$. By a quick calculation, one realizes that $\mathcal L_{b,c}$ is self-adjoint in the $\langle \cdot, \cdot\rangle$ inner product. If $0$ is an eigenvalue of $\mathcal L_{b,c,}$ (which must be simple, since we are working on the real line), this means the two solutions $Y_\infty, Y_{-\infty}$ to \eqref{e:Y-system} are parallel and can be chosen equal to each other after scaling by an appropriate constant. We let $Y= Y_\infty = Y_{-\infty}$, and normalize $Y$ so that $\langle Y,Y\rangle = 1$. From above, we have $\lim_{y\to \pm \infty} e^{-m|y|}Y(y) = c_\pm$ for some constants $c_\pm$, and $|Y(y)| \leq Ce^{-m|y|}$ for all $y$.

Let 
\[Y^\perp = \{f \in L^2(\R), \langle Y, f\rangle = 0\}.\] 
We also define the standard projection operators $P_Y f = \langle Y,f\rangle Y$ and $P_{Y^\perp} = f - P_Y f$. 
If there is a solution $U_\delta$ to \eqref{e:Udelta}, the self-adjointness of $\mathcal L_{b,c}$ implies \[0 = \langle Y, \mathcal L_{b,c} U_\delta\rangle = \langle Y, c\xi + \mathcal N(U_\delta)\rangle.\] Therefore, we search for solutions such that $c\xi + \mathcal N(U_\delta)\in Y^\perp$.  
The following lemma constructs an integral operator to invert $\mathcal L_{b,c}$, that is well-defined on $Y^\perp$.




\begin{lemma}\label{l:L-inverse}
As above, assume $Y$ is an eigenfunction of $\mathcal L_{b,c}$ with eigenvalue $0$. Then there exists an inverse operator $\mathcal L_{b,c}^{-1}:Y^\perp \to Y^\perp$ for $\mathcal L_{b,c}$, such that
\begin{equation}\label{e:inv-est}
 \|\mathcal L_{b,c}^{-1} \eta\|_{X} + \|\partial_y\mathcal L_{b,c}^{-1} \eta\|_X \leq C_0\|\eta\|_{X}, \quad \eta \in Y^\perp.
 \end{equation}
 If, in addition,  $\|\eta\|_{\sim} = \|e^{k|y|} \eta(y)\|_{L^\infty(\R)}< \infty$ for some $k\in (0,m)$, then
\begin{equation}\label{e:L-exp-bound}
 \|\mathcal L_{b,c}^{-1} \eta\|_{\sim} + \|\partial_y\mathcal L_{b,c}^{-1} \eta\|_{\sim} \leq C_0\|\eta\|_{\sim}.
 \end{equation}
\end{lemma}

\begin{proof}
We construct $\mathcal L_{b,c}^{-1}$ by solving, for given $\eta$, the equation $\mathcal L_{b,c} f = \eta$ on $[0,\infty)$ and $(-\infty,0]$ and patching the solutions together. The solution $f$ will be $C^1$ at $y=0$ exactly when $\eta$ is perpendicular to $Y$. In more detail, define
 \[ Z(y) := Y(y)\left[ 1 + \int_0^y \frac { e^{\int_0^w b(s) \dd s}}{Y(w)^2} \dd w\right], \quad y\in \R.\]
Note that $Z(0) = Y(0)$, $Z'(0) = Y'(0) + 1/Y(0)$,\footnote{By Sturm-Liouville theory, zeros of $Y$ are isolated, so we may assume $Y(0)\neq 0$ by shifting the variable $y$ by a small constant if necessary, which does not affect any of the hypotheses of Theorem \ref{t:near-constant}.} and $\mathcal L_{b,c} Z = 0$. The decay of $Y$ for large $|y|$ implies $|Z(y)|\leq Ce^{m|y|}$. By direct calculation, we also have $|Z'(y)| \leq C e^{m|y|}$.

 As above, the Wronskian $W_{Y,Z}$ of $Y$ and $Z$ will satisfy $W_{Y,Z}(y) = W_{Y,Z}(0) e^{\int_0^y b(s) \dd s} = e^{\int_0^y b(s) \dd s}$. For $y\in [0,\infty)$, we use the Green's function
 \[ G_+(y,w) : = \frac {1}{W_{Y,Z}(w)} \begin{cases} Y(y) Z(w) , &0\leq w<y,\\ Y(w) Z(y), &0\leq w\geq y.\end{cases}
\] 
Let us derive bounds for this Green's function in the norm $\|\cdot\|_{X_+} = \|\cdot\|_{L^1(\R_+)} + \|\cdot\|_{L^\infty(\R_+)}$. For any $\eta \in X_+$, we have
\[\begin{split}
 \int_0^\infty G_+(y,w) \eta(w) \dd w &\leq C \left( Y(y) \int_0^y Z(w) \eta(w) \dd w + Z(y)\int_y^\infty Y(w) \eta(w) \dd w\right)\\
&\leq C\|\eta\|_{L^\infty(\R_+)} \left( e^{-my} \int_0^y e^{mw} \dd w + e^{my} \int_y^\infty e^{-mw} \dd w \right)\\
&\leq C\|\eta\|_{L^\infty(\R_+)},
\end{split}\]
and
\[\begin{split}
\int_0^\infty &\left| \int_0^\infty G_+(y,w) \eta(w) \dd w \right| \dd y\\
 &\leq C \int_0^\infty\left( |Y(y)| \int_0^y |Z(w) \eta(w) |\dd w + |Z(y)|\int_y^\infty |Y(w)\eta(w)| \dd w\right)\dd y\\
&\leq C \int_0^\infty \left( \int_0^y e^{m(w-y)} |\eta(w)|\dd w + \int_y^\infty e^{m(y-w)} |\eta(w)| \dd w \right)\dd y\\
&\leq C\int_0^\infty \int_0^\infty |\eta(w)| e^{-m|y-w|}\dd y \dd w\\
&\leq C\|\eta\|_{L^1(\R_+)},
\end{split}\]
and we see 
\[\left\|\int_0^\infty G_+(y,w) \eta(w) \dd w \right\|_{X_+} \leq C \|\eta\|_{X_+}.\]
By direct calculation, $\mathcal L_{b,c} \int_0^\infty G_+(y,w) \eta(w) \dd w = \eta(y)$ for $y\in [0,\infty)$ and $\eta\in X_+$. Next, note that $Y'$ and $Z'$ share the same decay properties as $Y$ and $Z$ respectively, so a similar calculation after taking a $y$-derivative gives
\[ \left\|\partial_y \int_0^\infty G_+(y,w) \eta(w) \dd w\right\|_{X_+} \leq C\|\eta\|_{X_+}.\]

On the half-line $(-\infty,0]$, we use the Green's function
 \[ G_-(y,w) : = -\frac {1}{W_{Y,Z}(w)} \begin{cases} Y(y) Z(w) , &y<w\leq 0,\\ Y(w) Z(y), &w\leq y\leq 0.\end{cases}
\] 
Defining $\|\cdot \|_{X_-} = \|\cdot\|_{L^\infty(\R_-)} + \|\cdot \|_{L^1(\R_-)}$, calculations similar to above show $\mathcal L_{b,c} \int_{-\infty}^0 G_-(y,w) \eta(w) \dd w = \eta(y)$ for $y\in (-\infty,0]$, and 
\[\left\|\int_{-\infty}^0 G_-(y,w) \eta(w) \dd w \right\|_{X_-} +\left\|\partial_y\int_{-\infty}^0 G_-(y,w) \eta(w) \dd w \right\|_{X_-}  \leq C \|\eta\|_{X_-}.\]

Define 
\[(\mathcal M \eta)(y) = \begin{cases} \displaystyle\int_0^\infty G_+(y,w) \eta(w) \dd w, & y\geq 0,\\
\displaystyle\int_{-\infty}^0 G_-(y,w) \eta(w) \dd w, &y<0.\end{cases} \]
Now, if $\eta$ satisfies the orthogonality condition $\langle Y, \eta\rangle = \int_\R e^{-\int_0^y b(s) \dd s} Y(y) \eta(y) \dd y = 0$, then
\[ \begin{split}
\int_0^\infty G_+(0,w) \eta(w) \dd w &= \int_0^\infty \frac 1 {W_{Y,Z}(w)} Y(w) \eta(w) \dd w\\
&= \int_0^\infty e^{-\int_0^w b(s) \dd s} Y(w) \eta(w) \dd w\\
&= -\int_{-\infty}^0 e^{-\int_0^w b(s) \dd s} Y(w) \eta(w) \dd w\\
&= \int_{-\infty}^0 G_-(0,w) \eta(w) \dd w,
\end{split}
\]
and $\mathcal M\eta$ is continuous across $y=0$. A similar calculation after taking a $y$-derivative shows that $\mathcal M\eta$ is $C^1$ across $y=0$. From our estimates for $G_+$ and $G_-$, it is also clear that
\begin{equation}\label{e:M-est}
 \|\mathcal M \eta\|_{X} + \|\partial_y \mathcal M \eta\|_X \leq C\|\eta\|_{X}, \quad \eta \in Y^\perp.
 \end{equation}

The above calculations do not necessarily imply $\mathcal M\eta \in Y^\perp$ when $\eta \in Y^\perp$, so we define $\mathcal L_{b,c}^{-1} = P_{Y^\perp} \mathcal M$. With this definition, we have 
\[\mathcal L_{b,c} \mathcal L_{b,c}^{-1} \eta = \mathcal L_{b,c} ( \mathcal M\eta - \langle Y, \mathcal M\eta\rangle Y) = \eta, \]
since $\mathcal L_{b,c} Y = 0$. The projection $P_{Y^\perp}$ is bounded on $L^2(\R)$, and since $Y$ lies in the space $X$ (in fact, it is bounded and exponentially decaying, as explained above), the estimate \eqref{e:inv-est} follows from \eqref{e:M-est}.

Finally, if $\|\eta\|_{\sim}<\infty$, then a direct calculation that is essentially the same as the proof of \eqref{e:exp-bound} applied to $G_+$ and $G_-$ establishes the bound \eqref{e:L-exp-bound}. We omit the details.
\end{proof}

 We would like to convert the equation $\mathcal L_{b,c} U_\delta = c\xi + \mathcal N(U_\delta)$ into integral form, as in Case 1. However, for a general $U_\delta\in L^2(\R)$, it is not at all clear that $c\xi + \mathcal N(U_\delta)$ lies in $Y^\perp$, which is required for our inverse operator $\mathcal L_{b,c}^{-1}$ to be well-behaved. Proceeding formally for the moment, we write $U_\delta = \eta + \alpha Y$, where $\eta\in Y^\perp$ and $\alpha \in \R$. Our equation becomes
 \[ \mathcal L_{b,c} \eta = c\xi + \mathcal N(\eta + \alpha Y),\]
since $\mathcal L_{b,c}Y = 0$. This suggests the following scheme: for given $\eta\in Y^\perp$, choose $\alpha(\eta)$ such that $c\xi + \mathcal N(\eta +\alpha(\eta) Y) \in Y^\perp$. If such an $\alpha(\eta)$ exists, then we may define the integral operator 
\[ \tilde{\mathcal T}\eta := \mathcal L_{b,c}^{-1}[c\xi + \mathcal N(\eta + \alpha(\eta)Y)],\]
and seek a fixed point $\eta$ for $\tilde{\mathcal T}$. 

To define the class of functions in which $\eta$ will live, first we claim that the quantity $\langle Y,Y^3\rangle$ is bounded below by a positive constant $c_1$ that is independent of $\delta$. Indeed, since $Y\leq Ce^{-m|y|}$, there exists $M>0$ depending on $C$ and $m$ such that $\int_{|y|>M} \omega Y^2 \dd y \leq \frac 1 2$. We recall that $C$ depends on $b$ and $F$, but can be taken independent of $c(y)$ and therefore of $\delta$. 

Next, since $\langle Y, Y\rangle = 1$, we have
\[ \frac 1 2 \leq \int_{-M}^M \omega Y^2 \dd y \leq \left(\int_{-M}^M \omega Y^4 \dd y\right)^{1/2} (2M)^{1/2}.\]
This implies $\langle Y, Y^3\rangle \geq \int_{-M}^M \omega Y^4 \dd y\geq c_1>0$, with $c_1$ depending on $C$ and $m$, as claimed.

Now, let $K_0>1$ be such that
\begin{equation}\label{e:K0}
 \left|\frac {24 \langle Y, c\xi\rangle}{F^{(4)}(\xi)c_1} \right|\leq K_0 \delta.
 \end{equation}
 The key point is that $K_0$ is chosen independently of $\delta$. With $C_0$ the constant from Lemma \ref{l:L-inverse}, let $K_1>1$ be such that
\begin{equation}\label{e:K1}
C_0\left(\left\|c\xi \right\|_X + \frac 1 3  K_0\delta \|F\|_{C^4(\R)}\left\|  Y\right\|_{L^\infty}\right) \leq K_1\delta.
\end{equation}
 The choices of $K_0$ and $K_1$ will be justified below.
 
Next, define the class
\[ \tilde{\mathcal A}_\delta := \{ \eta: \|\eta\|_X \leq 2K_1 \delta\}.\]
Our next lemma shows that $\alpha(\eta)$ exists and can be chosen uniquely, for any $\eta\in \tilde{\mathcal A}_\delta$.
\begin{lemma}\label{l:orthogonality}
If $\delta>0$ is sufficiently small (depending on $K_0$, $K_1$, and the potential $F$) then for any $\eta \in \tilde{\mathcal A}_\delta$, there is a unique $\alpha\in [-(2K_0\delta)^{1/3}, (2K_0\delta)^{1/3}]$ such that 
\begin{equation}\label{e:g-alpha}
\begin{split}
 0 &= \langle Y, c\xi + \mathcal N(\eta + \alpha Y)\rangle = \langle Y, c\xi\rangle + \langle Y, F''(\xi) ( \eta + \alpha Y) - F'(\xi+\eta + \alpha Y) \rangle.
\end{split}
\end{equation}
The map $\eta \mapsto \alpha$ is continuous in the following sense: for any sequence $\{\eta_k\}\subset \tilde{\mathcal A}_\delta$ such that $\eta_k \to \eta_0$ a.e., there holds $\alpha(\eta_k) \to \alpha(\eta_0)$.
\end{lemma}
\begin{proof}
%
%

%
%
%

We rewrite formula \eqref{e:g-alpha}, using $F'''(\xi) = 0$, as
\begin{equation}\label{e:poly}
\begin{split}
0 &= \langle Y, c\xi\rangle - \frac 1 6 \langle Y, F^{(4)}(z_y)(\eta + \alpha Y)^3\rangle\\
&=  \langle Y, c\xi\rangle - \frac 1 6 \langle Y, F^{(4)}(z_y)\eta^3\rangle - \frac 1 2 \alpha \langle Y, F^{(4)}(z_y) Y \eta^2\rangle\\
&\quad   -\frac 1 2 \alpha^2 \langle Y, F^{(4)}(z_y) Y^2 \eta\rangle -\frac 1 6 \alpha^3 \langle Y,F^{(4)}(z_y) Y^3\rangle ,
\end{split}
\end{equation}
where $z_y \in \R$ satisfies $|z_y - \xi|\leq |\eta +\alpha Y| \lesssim K_1\delta + K_0 \delta^{1/3}$. 
Choosing $\delta$ small enough, depending on $K_0$, $K_1$, and $\|F\|_{C^4(\R)}$, we ensure 
\[\frac 1 2 F^{(4)}(\xi) \leq F^{(4)}(z_y) \leq 2 F^{(4)}(\xi),\]
for all $y$. 

Denoting the right-hand side of \eqref{e:poly} by $g(\alpha)$, we want to select an interval on which $g$ has a sign change. 
For $\eta, \alpha$ with $|\alpha|\leq (2K_0 \delta)^{1/3}$ and $\|\eta\|_X \leq 2K_1  \delta$, we have, using the definition \eqref{e:K0} of $K_0$,
\begin{equation}\label{e:rhs}
 \begin{split}
\frac 6 {|\langle Y, F^{(4)}(z_y) Y^3\rangle|}& \left| \langle Y,c\xi\rangle - \frac 1 6 \langle Y, F^{(4)}(z_y)\eta^3\rangle - \frac 1 2 \alpha \langle Y, F^{(4)}(z_y) Y \eta^2\rangle  -\frac 1 2 \alpha^2 \langle Y, F^{(4)}(z_y) Y^2 \eta\rangle\right| \\
&\leq\frac 1 2 K_0\delta + C\left((K_1\delta)^3 + K_0^{1/3} K_1^2\delta^{7/3} + K_0^{2/3} K_1\delta^{5/3}\right)
\end{split}
\end{equation}
for a constant $C$ depending on $\|F\|_{C^4(\R)}$ and $\|Y\|_{L^\infty(\R)}$. 
If we impose the condition
\[\delta< \frac{ K_0^{1/2}}{(6C K_1)^{3/2}},\] 
then the last expression in \eqref{e:rhs} is bounded above by 
\[ \frac 1 2 K_0 \delta(1 + C \delta) < K_0\delta,\]
if we additionally impose $\delta < 1/C$.  
Therefore, for $\delta$ satisfying our constraints and $\alpha = \pm (2K_0 \delta)^{1/3}$, the $\alpha^3$ term in $g(\alpha)$ is strictly larger in magnitude than the other terms combined, and $g$ changes sign on the interval
\[ I_\delta := [-(2K_0\delta)^{1/3}, (2K_0\delta)^{1/3}],\]
as claimed.

Next, we claim that $g$ is monotonic on $I_\delta$. First, assume $\{y :\eta(y)\neq0\}$ has positive measure. Note that
\begin{equation}\label{e:g-prime}
\begin{split}
g'(\alpha) &= -\frac 1 2 \langle Y, F^{(4)}(z_y) Y\eta^2\rangle - \alpha \langle Y, F^{(4)}(z_y) Y^2\eta\rangle - \frac 1 2 \alpha^2 \langle Y, F^{(4)}(z_y) Y^3\rangle,\\
&= -\frac 1 2 \langle Y, F^{(4)}(z_y) Y (\eta+\alpha Y)^2\rangle,
\end{split}
\end{equation}
which is never zero unless $\eta = -\alpha Y$ almost everywhere, which is impossible since $\eta$ is not zero almost everywhere, and $\eta\in Y^\perp$. We conclude $g$ is monotonic and therefore has a unique root in $I_\delta$.

On the other hand, if $\eta = 0$ almost everywhere, the unique root for $g(\alpha)$ is explicitly given by 
\[ \alpha = \left[\frac{6\langle Y, c\xi\rangle}{\langle Y, F^{(4)}(z_y) Y^3\rangle}\right]^{1/3},\]
and we have $|\alpha|\leq (2K_0\delta)^{1/3}$ in this case as well.

For the continuity, let $\{\eta_k\}$ be a sequence in $\tilde{\mathcal A}_\delta$ with $\eta_k\to \eta_0$ a.e., as in the statement of the lemma. Since $Y$ decays exponentially and $\eta_k$ are uniformly bounded in $L^\infty$, dominated convergence implies
\begin{equation}\label{e:eta-k}
\begin{split}
 \langle Y,F^{(4)}(z_y) \eta_k^3\rangle &\to \langle Y,F^{(4)}(z_y) \eta_0^3\rangle,\\
 \langle Y, F^{(4)}(z_y) Y \eta_k^2\rangle &\to   \langle Y, F^{(4)}(z_y) Y \eta_0^2\rangle,\\
  \langle Y, F^{(4)}(z_y) Y^2\eta_k\rangle &\to   \langle Y, F^{(4)}(z_y) Y^2\eta_0\rangle.
\end{split} 
\end{equation}
Considering the two cases above, if $\{y:\eta_0(y)\neq 0\}$ has positive measure, then the same is true of $\eta_k$ for $k$ sufficiently large. For each $k$, $\alpha_k$ is the unique simple root in the interval $I_\delta$, and simple roots of polynomials depend continuously on the coefficients. We conclude from \eqref{e:eta-k} that $\alpha(\eta_k) \to \alpha(\eta_0)$. On the other hand, if $\eta_0(y) = 0$ almost everywhere, \eqref{e:eta-k} and the uniform bounds $|\alpha_k|\leq (2K_0\delta)^{1/3}$ imply
\[\begin{split}
 \alpha_k &= \left(\frac 1 {\langle Y,F^{(4)}(z_y) Y^3\rangle} \left[ 6 \langle Y, c\xi\rangle - \langle Y, F^{(4)}(z_y)\eta_k^3\rangle \right.\right.\\
 &\left.\left. \qquad \qquad- 3\alpha_k \langle Y, F^{(4)}(z_y) Y \eta_k^2\rangle - 3\alpha_k^2 \langle Y, F^{(4)}(z_y) Y^2 \eta_k\rangle  \right]\right)^{1/3}\\
 &\to \left(\frac {6 \langle Y, c\xi\rangle} {\langle Y,F^{(4)}(z_y) Y^3\rangle} \right)^{1/3} = \alpha_0,
 \end{split}\]
 as claimed.
\end{proof}


With $K_0>0$ as in \eqref{e:K0}, define the integral operator 
 \[ \tilde {\mathcal T}: \eta \mapsto \mathcal L_{b,c}^{-1}\left[ c\xi +\mathcal N(\eta+\alpha(\eta) Y)\right],  \qquad \|\eta\|_X \leq 2K_1 \delta,\]
 where $\alpha = \alpha(\eta)$ is the number provided by Lemma \ref{l:orthogonality}. Since $c\xi + \mathcal N(\eta + \alpha Y) \in Y^\perp$, Lemma \ref{l:L-inverse} implies
\[\left\|\tilde{\mathcal T}\eta\right\|_{X} \leq C_0\left\|c\xi+ \mathcal N(\eta+\alpha Y)\right\|_X,\]  
 for a constant $C_0>0$ independent of $\eta$ and $\delta$. Rewriting $\mathcal N(\eta + \alpha Y) =  -F'(\xi+\eta + \alpha Y) - F''(\xi) (\eta+\alpha Y)$ using $F'''(\xi) = 0$, we have, for some function $h:\R\to \R$ depending on $\eta$ and $\alpha$, 
 \begin{equation}\label{e:T-ineq}
  \begin{split}
\left\| \tilde{\mathcal T} \eta\right\|_X &\leq C_0 \left\|c\xi  - \frac 1 6 F^{(4)}(h)(\eta+\alpha Y)^3\right\|_X\\
 &\leq  C_0\left\|c\xi  - \frac 1 6 F^{(4)}(h) \alpha^3 Y^3\right\|_X +C_0\left\|\frac 1 6 F^{(4)}(h)(\eta^3+3\eta^2 \alpha + 3 \eta \alpha^2)\right\|_X.
 \end{split}
 \end{equation} 
Since $\eta \in \tilde{\mathcal A}_\delta$, Lemma \ref{l:orthogonality} implies $|\alpha|\leq (2K_0 \delta)^{1/3}$. Therefore, for the first term on the right in \eqref{e:T-ineq}, we have
\begin{equation}\label{e:C0ref}
C_0\left\|c\xi  - \frac 1 6 F^{(4)}(h) \alpha^3 Y^3\right\|_X \leq K_1 \delta,
\end{equation}
by our definition \eqref{e:K1} of $K_1$.  With \eqref{e:T-ineq}, we now have
\[ \|\tilde{\mathcal T} \eta\|_X \leq K_1 \delta + CC_0 \left(K_1^3\delta^3 + K_0^{1/3}K_1^2 \delta^{7/3} + K_0^{2/3}K_1 \delta^{5/3}\right) .\] 
Now we impose the further condition
\[ \delta < \frac 1 {(6CC_0)^{3/2} K_0 K_1},\]
which ensures 
\[ \|\tilde{\mathcal T}\eta\|_X \leq 2K_1 \delta,\quad \eta\in \tilde{\mathcal A}_\delta,\]
or equivalently, $\tilde{\mathcal T}$ maps $\tilde{\mathcal A}_\delta$ into itself.

The map $\tilde {\mathcal T}$ is not necessarily a contraction on $\tilde {\mathcal A_\delta}$, but we can still show that iterates of $\tilde {\mathcal T}$ converge: define
\begin{equation}\label{e:eta-k-def}
 \begin{split}
\eta_0(y) &\equiv 0,\\
\eta_{k+1}(y) &= \tilde {\mathcal T}(\eta_k)(y) = \mathcal L_{b,c}^{-1}\left[ c\xi +\mathcal N(\eta_k+\alpha(\eta_k) Y)\right], \quad k= 0,1,2,\ldots.
\end{split} 
\end{equation}
Our work above shows that $\eta_k\in \tilde{\mathcal A}_\delta$ for all $k$. From Lemma \ref{l:L-inverse}, we have
\[\begin{split}
 \|\eta_{k+1}\|_X + \|\partial_y \eta_{k+1}\|_X &\leq C \left\|\mathcal L_{b,c}^{-1} \left[ c\xi +\mathcal N(\eta_k+\alpha(\eta_k) Y)\right]\right\|_X \\
 &\leq C \left(\|c\xi \|_X+ \|\mathcal N(\mathcal \eta_k + \alpha(\eta_k)Y)\|_X\right)\\
 &\leq C \delta,
 \end{split}\]
exactly as in \eqref{e:T-ineq}, since $\eta_k \in \tilde{\mathcal A}_\delta$. This uniform $C^1$ estimate shows that $\{\eta_k\}$ is a precompact sequence, and some subsequence converges uniformly on compact subsets of $\R$ to a function $\eta_0 \in \tilde{\mathcal A}_\delta$. 
In particular, Lemma \ref{l:orthogonality} implies $\alpha(\eta_k) \to \alpha(\eta_0)$. 
Returning to the defining equation for $\eta_{k+1}$, let $G_+$ and $Z$ be as in the proof of Lemma \ref{l:L-inverse}. For $y\geq 0$, we have
\[\begin{split}
 \eta_{k+1}(y) &= \int_0^\infty G_+(y,w) \left[c\xi + \mathcal N(\eta_k+\alpha(\eta_k) Y)\right](w) \dd w\\
 &=  Y(y) \int_0^y \frac{Z(w)}{W_{Y,Z}(w)} \left[c\xi + \mathcal N(\eta_k+\alpha(\eta_k) Y)\right](w) \dd w \\
 &\quad + Z(y)\int_y^\infty \frac{Y(w)}{W_{Y,Z}(w)} \left[c\xi + \mathcal N(\eta_k+\alpha(\eta_k) Y)\right](w) \dd w
 \end{split} \]
On the left side of this equality, we clearly have $\eta_{k+1}(y)\to \eta_0(y)$ for each $y$. We can take the limit in the right side via dominated convergence because $\alpha(\eta_k)\to \alpha(\eta_0)$, and because of the bounds $|Y|\leq Ce^{-m|y|}$, $|Z|\leq Ce^{m|y|}$, and $\mathcal N(\eta_k + \alpha(\eta_k)Y) \leq C|\eta_k +\alpha(\eta_k)Y|^3 \leq C(2K_1 \delta + (2K_0\delta)^{1/3})^3$. The same argument with $G_-$ replacing $G_+$ applies when $y<0$. We conclude
\begin{equation}\label{e:fixed-point}
 \eta = \mathcal L_{b,c}^{-1} [c\xi + \mathcal N(\eta + \alpha(\eta)Y)].
 \end{equation}
 By direct calculation, this right-hand side is differentiable in $y$, and we conclude $\eta\in C^1$. Since $\mathcal L_{b,c} Y = 0$, we have
\[ \mathcal L_{b,c}(\eta + \alpha(\eta)Y) = c\xi + \mathcal N(\eta+\alpha(\eta)Y),\]
and letting $U_\delta = \eta + \alpha(\eta) Y$, we have constructed a solution to \eqref{e:Udelta}.

 Next, under the additional assumption that $c$ has exponential decay proportional to $e^{-k|y|}$ for some $k < m$, we return to our definition of the class $\tilde{\mathcal A}_\delta$ and replace $K_1$ from \eqref{e:K1} with a constant $\tilde K_1$ such that 
 \[ C_0\left( \|c\xi\|_{\sim} + \frac 1 3 K_0 \delta \|F\|_{C^4(\R)} \|Y\|_{\sim}\right) \leq \tilde K_1 \delta.\]
(Recall that $|Y(y)|\leq C e^{-m|y|}$.) Letting $\tilde {\mathcal A}_{\sim} = \{\eta: \|\eta\|_{\sim} \leq 2\tilde K_0 \delta\}$, arguing exactly as in \eqref{e:C0ref} shows that $\tilde{\mathcal T}$ maps $\tilde{\mathcal A}_{\sim}$ into itself. Therefore, the iterates $\eta_k$ defined in \eqref{e:eta-k-def} satisfy a uniform bound in the $\|\cdot\|_{\sim}$, which is passed to $\eta_0$ via pointwise convergence. The exponential decay of $\eta'$ follows from differentiating formula \eqref{e:fixed-point}. This completes the proof of Theorem \ref{t:near-constant}.

\section{Proof of stability}

\subsection{Orbital stability}


It is more or less standard that, if $b$, $c$, and $F$ are such that the energy functional $E(u(t))$ controls the $H^1(\R)\times L^2(\R)$ norm of $u(t)$, solutions starting close to a constant state $u \equiv \xi$ exist for all time and remain close to $\xi$. We give a proof for the convenience of the reader.

\begin{proposition}\label{p:orbital-constant}
Assume there exists $\mu>0$ such that 
\[c(y)\leq F''(\xi) -\mu, \quad y\in \R.\] 
Then there exist $C,\eps_0>0$ such that if $\|u(0)-\xi\|_{H^1(\R)} + \|\partial_t u(0)\|_{L^2(\R)} < \eps$ with $\eps\leq \eps_0$, then the corresponding solution $u(t,x)$ exists for all time, with 
\[\|u(t)\|_{H^1(\R)} + \|\partial_t u(t)\|_{L^2(\R)} \leq C\eps.\]
\end{proposition}


\begin{proof}
Recall that the energy \eqref{e:energy}  
is conserved along the flow of \eqref{e:main-y}. Since $c\in L^\infty$ and $F\in C^2$, we have 
\begin{equation}\label{e:C1}
  E(u(t))  \leq C_1 (\|u(t)\|_{H^1(\R)}^2 + \|\partial_t u(t)\|_{L^2(\R)}^2),
  \end{equation}
for any $t$ in the time domain of $u$. On the other hand, a Taylor expansion for $F$ shows that
\[ \frac 1 2 (F(u) - cu^2) \geq \frac 1 2(F''(\xi)-c)u^2 - K u^3 \geq \frac \mu 2 u^2 - K u^3,\]
for any $y\in \R$, where $K>0$ depends on the $C^3$ norm of $F$. By Sobolev embedding, $\|\psi\|_{L^\infty}\leq C_0\|\psi\|_{H^1}$, so for any $t$ with $\|u(t)\|_{H^1}\leq r_0 := \mu/(2C_0K)$, we have
\begin{equation}\label{e:C2}
\|u(t)\|_{H^1(\R)}^2 + \|\partial_t u(t)\|_{L^2(\R)}^2 \leq C_2 E(u(t)), \quad \text{ if } \|u(t)\|_{H^1(\R)} \leq r_0.
\end{equation}

Now, let $\eps_0 := \sqrt{ r_0/(1+C_1C_2)}$, and assume $\|u(0)\|_{H^1(\R)} + \|\partial_t u(0)\|_{L^2(\R)} < \eps$ for $\eps< \eps_0$, as in the statement of the theorem. Let $E_0 = E(u_0)$ be the energy of $(u(0),\partial_t u(0))$. Thanks to \eqref{e:C1} applied at $t=0$, we must have $E_0 \leq C_ 1\eps^2$. By a standard fixed-point argument, a solution to \eqref{e:main-y} exists on $[0,T]\times\R$ for some $T>0$ depending only on $\eps$. 


We claim $\|u(t)\|_{H^1(\R)} + \|\partial_t u(t)\|_{L^2(\R)} < r_0$ for all $t\in [0,T]$. If not, let $t_0$ be the first time where $\|u(t_0)\|_{H^1(\R)} = r_0$. By energy conservation and \eqref{e:C2}, we have
 \[ \| u(t_0)\|_{H^1}^2 + \|\partial_t u(t_0)\|_{L^2}^2 \leq C_2 E_0 \leq C_1 C_2 \eps^2 < r_0,\]
a contradiction.

Therefore, the solution can be extended up to a time $T + T_1$, with $T_1>0$ depending on $r_0$. Repeating this argument, we extend the solution for all time.
\end{proof}

We can also address the orbital stability of the near-constant states $U(y) = \xi + U_\delta(y)$ constructed in Theorem \ref{t:near-constant}, in the case $F''(\xi) >0$. Using $F(\xi) = F'(\xi) = 0$ and $\|U_\delta\|_X + \|\partial_y U_\delta\|_X \lesssim \delta$, it is not hard to see that the energy of $U(y)$ is finite and bounded above by a constant times $\delta$. 

Furthermore, if $\delta>0$ is sufficiently small, then initial data which is $\delta$-close to $(\xi,0)$ in the $H^1\times L^2$ norm also satisfies the hypotheses of Proposition \ref{p:orbital-constant}. This implies near-constant states are also orbitally stable, since solutions staying close to the constant state $\xi$ are also $\delta$-close to $U(y)$.

\subsection{Asymptotic stability} Next, we address the much more nontrivial property of asymptotic stability.




For the time being, we let $U(y)$ denote either the constant state $U\equiv \xi$ or the near-constant steady state guaranteed by Theorem \ref{t:near-constant}. Let $v(t,y) = u(t,y)-U(y)$ be the perturbation. By Proposition \ref{p:orbital-constant} (orbital stability), taking $v(0)$ less than $\eps$ in the $H^1\times L^2$ norm is enough to ensure
\[ \|v(t)\|_{H^1} + \|\partial_t v(t)\|_{L^2} \lesssim \eps, \quad t\in [0,\infty).\]
We see that $v$ satisfies
\[ \partial_t^2 v - [\partial_y^2 v + b\partial_y v + cv] = F'(U) - F'(U+v) =  \mathcal N(U,v),\]
with $\mathcal N(U,v) = F'(U)-  F'(U+v)$. 
For convenience, we rewrite this as a first-order system, with $(v_1,v_2) = (v,\partial_t v)$:
\begin{equation}\label{e:vector}
\begin{split}
\partial_t v_1 &= v_2,\\
\partial_t v_2 &= \partial_y^2 v_1 + b\partial_y v_1 + cv_1 +\mathcal N(U,v_1).
\end{split}
\end{equation}

Let $\psi(y) = \lambda\tanh(y/\lambda)$ for some $\lambda>0$ to be chosen later. Define the Virial functional 
\begin{equation}\label{e:Idef}
 \mathcal I(v) := \int_\R \psi(\partial_y v_1) v_2 \dd y + \frac 1 2 \int_\R \psi' v_1 v_2 \dd y= \int_\R\left( \psi \partial_y v_1 + \frac 1 2 \psi' v_1\right) v_2 \dd y.
 \end{equation}
For any $h\in H^1(\R)$, we have the identity
\[ \int_\R \left(\psi \partial_y h + \frac 1 2 \psi' h\right) h \dd y = 0.\]
Using this identity with $h = v_2$, we differentiate $\mathcal I(v)$ for $(v_1,v_2)$ a solution of \eqref{e:vector}, and integrate by parts several times:
\begin{equation}\label{e:ddtI}
\begin{split}
\frac d {dt} \mathcal I(v) &= \int_\R \left(\psi \partial_y v_1 + \frac 1 2 \psi' v_1\right) \partial_t^2 v_1 \dd y \\
&= \int_\R \left(\psi \partial_y v_1 + \frac 1 2 \psi' v_1\right) \left(\partial_y^2 v_1 + b\partial_y v_1 + c v_1 +\mathcal N(U,v_1)\right) \dd y \\
&=\int_\R \left( - \psi' (\partial_y v_1)^2 - \frac 1 2 \psi'' v_1 \partial_y v_1 + \left( \psi \partial_y v_1 + \frac 1 2 \psi' v_1\right)(b\partial_y v_1 + cv_1+\mathcal N(U,v_1))\right) \dd y \\
&= - \int_\R \psi' (\partial_y v_1)^2 + \frac 1 4 \int_\R \psi''' v_1^2 + \int_\R \left(\psi b (\partial_y v_1)^2 + \frac 1 2 \psi' c v_1^2 \right)\dd y\\
&\quad + \int_\R \left( \frac 1 2 \psi' b + \psi c\right) (\partial_y v_1) v_1 \dd y + \int_\R \left( \psi \partial_y v_1 + \frac 1 2 \psi' v_1\right)\mathcal N(U,v_1) \dd y.
\end{split}
\end{equation}
We also define 
\[ \zeta(y) := \sqrt{\psi'(y)} = \sech(y/\lambda), \quad w : = \zeta v_1.\]
It will be convenient to work with the weighted norms
\[ \|h\|_{L^2_\omega}^2 := \int_\R h^2 \sech(y) \dd y, \quad \|h\|_{H^1_\omega}^2 := \int_\R ((\partial_y h)^2 + h^2 ) \sech(y) \dd y.\]
Let us define the bilinear form
\begin{equation}\label{e:B}
\mathcal B(v) := \int_\R \psi' (\partial_y v)^2 \dd y - \frac 1 4 \int_\R \psi''' v^2 \dd y.
\end{equation}
Then the first two terms on the right in \eqref{e:ddtI} are equal to $-\mathcal B(v_1)$. In the case of odd perturbations, one could use the coercivity of $\mathcal B$ for odd functions, as established in \cite[Lemma 2.1]{KMM2017breathers}: for any $\lambda>0$,
\begin{equation}\label{e:odd-coercivity}
 v \text{ odd}, v \in H^1 \quad \Rightarrow \quad  \mathcal B(v) \geq \frac 3 4 \int_\R (\partial_y w)^2 \dd y, \,   \text{ where } w(y) = \sech(y/\lambda) v(y).
 \end{equation}  
In the general case, we must analyze the coercivity properties of $\mathcal B$ for general (not necessarily odd) functions in $H^1(\R)$. We find that $\mathcal B$ is coercive only up to a zeroth-order correction term:
\begin{lemma}\label{l:B}
With $\mathcal B$ defined as in \eqref{e:B}, one has for any $v\in H^1$,
\[ \mathcal B(v) \geq \int_\R \left((\partial_y w)^2- \frac 1 {2\lambda^2}  \sech^2(y/\lambda) w^2  \right)  \dd y   ,\]
with $w = \zeta v$.
\end{lemma}

\begin{proof}
With $\zeta = \sqrt{\psi'} > 0$ and $w = \zeta v$, we have, by a direct calculation (see also \cite{KMM2017breathers}),
\[ \mathcal B(v) = \int_{\R} \left((\partial_y w)^2 + \frac 1 2 \left( \frac{\zeta''}\zeta - \frac {(\zeta')^2}{\zeta^2}\right) w^2 \right)\dd y = \int_\R \left((\partial_y w)^2 - \frac 1 {2\lambda^2}\sech^2(y/\lambda) w^2\right)\dd y.\]
\end{proof}

Later, we will see that the coercive terms coming from the coefficients $b$ and $c$ will counterbalance the negative term in Lemma \ref{l:B}.

%
%


Our next step is to analyze the term involving $\mathcal N(U,v_1) = F'(U) - F'(U+v_1)$ in \eqref{e:ddtI}. For this, we need the following auxiliary lemma, which extends \cite[Formula (2.22)]{KMM2017breathers} to the case of general (not necessarily odd) functions $v$:
\begin{lemma}\label{l:psi-prime}
For any $v\in H^1(\R)$ and $\lambda >2$, let $w = \zeta v = \sech(y/\lambda) v$. Then for any $q>0$, there holds
\begin{equation}\label{e:psi-prime}
\int_\R \psi' |v|^{2+q} \dd y \lesssim \lambda^2 \|v\|_{L^\infty}^q \|\partial_y w\|_{L^2}^2.
\end{equation}
\end{lemma}
\begin{proof}
With $\zeta = \sqrt{\psi'} = \sech(y/\lambda)$ and $v = w/\zeta$, we have
\begin{equation}\label{e:first-line}
 \begin{split}
 \int_\R \psi' |v|^{2+q} \dd y &\leq  \int_\R (\psi')^{-q/2} |w|^{2+q} \dd y = \int_\R \zeta^{-q} |w|^{2+q}\dd y \lesssim \int_\R e^{q|y|/\lambda} |w|^{2+q} \dd y. 
 \end{split}
 \end{equation}

Integrating by parts separately on $(-\infty,0]$ and $[0,\infty)$, we have
\[\begin{split}
 \int_\R e^{q|y|/\lambda} |w|^{2+q} \dd y &= \int_{-\infty}^0 e^{-qy/\lambda} |w|^{2+q} \dd y + \int_0^\infty e^{qy/\lambda} |w|^{2+q} \dd y\\
  &=  \frac {\lambda} q \int_{-\infty}^0 e^{-qy/\lambda} \partial_y (|w|^{2+q})\dd y  -\frac {\lambda} q |w(0)|^{2+q}\\
  &\quad  - \frac {\lambda} q \int_0^\infty e^{qy/\lambda} \partial_y (|w|^{2+q})\dd y - \frac {\lambda} q |w(0)|^{2+q}\\
  &= - \frac {\lambda} q \left( (2+q)\int_\R \sgn(y) e^{q|y|/\lambda} |w|^{q} w \partial_y w \dd y - 2 |w(0)|^{2+q}\right).
 \end{split}\]
For $w\in H^1(\R)\subset C^{1/2}(\R)$, the value of $w(0)$ is well-defined. Since $\sech(0) = 1$, we have 
\[|w(0)|^{2+q} = v(0)^{q} w(0)^2\leq \|v\|_{L^\infty(\R)}^{q} \|w\|_{H^1(\R)}^2,\] 
by the Sobolev embedding $H^1(\R) \subset L^\infty(\R)$. Using 
 $|w|^q = |w|^{q/2}\zeta^{q/2} |v|^{q/2} \lesssim e^{-q|y|/(2\lambda)} |w|^{q/2} |v|^{q/2}$, 
we have
\[ \begin{split}
\int_\R e^{q|y|/\lambda} |w|^{2+q} \dd y 
&\leq C \lambda \int_\R \|v\|_{L^\infty}^{q/2} e^{q|y|/(2\lambda)} |w|^{1+q/2} |\partial_y w| \dd y +  C\lambda \|v\|_{L^\infty(\R)}^q \|w\|_{H^1(\R)}^2\\
&\leq C^2 \lambda^2 \int_\R \|v\|_{L^\infty}^q |\partial_y w|^2 \dd y + \frac 1 2 \int_\R e^{q|y|/\lambda} |w|^{2+q} \dd y + C \lambda\|v\|^q_{L^\infty} \|w\|_{H^1}^2,
\end{split}\]
which implies
\[ \int_\R e^{q|y|/\lambda} |w|^{2+q} \dd y \leq C \lambda^2\|v\|_{L^\infty}^q \|w\|_{H^1}^2.\]
With \eqref{e:first-line}, we obtain
\[ \int_\R \psi' |v|^{2+q} \dd y \lesssim \lambda^2\|v\|_{L^\infty}^q \|w\|_{H^1}^2,\]
as desired. 
\end{proof}
Now we are able to prove the necessary integral bound for $\mathcal N(U,v_1)$:
\begin{lemma}\label{l:Fterm}
In the case $c\equiv 0$, $U\equiv 0$, and also in the case $c\not\equiv 0$, $F(u) = 1-\cos u$, there holds
\begin{equation}\label{e:J} \left|\int_\R (F'(U)- F'(U+v_1))(\psi \partial_y v_1 + \frac 1 2 \psi' v_1) \dd y \right| \lesssim \lambda^2( \eps \|\partial_y w\|_{L^2}^2 + \delta \|v_1\|_{H^1_\omega}^2). 
\end{equation}
\end{lemma}
\begin{proof}
For brevity, let $J$ refer to the left-hand side of \eqref{e:J}.

\medskip

\emph{Case 1: $c(y) \equiv 0$ and $U(y) \equiv 0$.} Then, by our assumption $F'(0) = 0$, we have 
\[ \begin{split} 
J = -\int_\R F'(v_1) (\psi \partial_y v_1 + \frac 1 2 \psi' v_1) \dd y &= -\int_\R (\psi \partial_y (F(v_1)) + \frac 1 2 \psi' v_1 F'(v_1))\dd  y \\
&= \int_\R (\psi' (F(v_1) - \frac 1 2 v_1 F'(v_1)) \dd y.
\end{split} \]
Since $F\in C^3$ and $F(0) = F'(0) = 0$, a Taylor expansion shows 
\begin{equation*}
\begin{split}
 |F(v_1) - \frac 1 2 v_1 F'(v_1)| &= \Big|F(0) + F'(0)v_1 + \frac 1 2 F''(0)v_1^2 + F'''(z_1)v_1^3 - \frac 1 2 F'(0) v_1 \\
 &\qquad - \frac 1 2 v_1^2 F''(0) - \frac 1 4 v_1^3 F'''(z_2)\Big|\\
  &\lesssim |v_1|^{3}.
  \end{split}
 \end{equation*}
 We therefore have 
 \[ |J|\leq \int_\R \psi' |v_1|^3 \dd y,\]
 and we may apply Lemma \ref{l:psi-prime} to obtain the conclusion of the lemma.
 
\medskip

%
%
%
\noindent \emph{Case 2: $F'(u) = \sin(u)$.} We now have $F'(U) - F'(U+v_1) = \sin U (1-\cos v_1) - \cos U \sin v_1$, and $F''(0) = 1$, so that, by Theorem \ref{t:near-constant},
\[ |\sin U(y)| \leq |U(y)| \lesssim \delta e^{-(1-\delta)|y|}, \quad |U'(y)| \lesssim \delta e^{-(1-\delta)|y|}.\]
We write $J = J_1 + J_2$. The first term is defined by
 \[ \begin{split}
J_1 &:= \int_\R \left(\psi \partial_y v_1 + \frac 1 2 \psi' v_1\right) \sin U (1-\cos v_1) \dd y.
\end{split}
\]
Since $|1-\cos v_1|\lesssim v_1^2$, the second part of $J_1$ is bounded by $\lesssim \int_\R \psi' |v_1|^3 \dd y$. For the first part,
\[\begin{split}
\int_\R \psi \partial_y v_1 \sin U(1-\cos v_1)\dd y &= \int_\R \psi \sin U \partial_y (v_1 - \sin v_1) \dd y\\
&= -\int_\R \psi' \sin U (v_1 - \sin v_1) \dd y - \int_\R \psi (\cos U) U' (v_1 - \sin v_1) \dd y.
\end{split}\]
Since $U$ decays at a faster rate than $\psi'(y) = \sech^2(y/\lambda)>0$, we have $|U'| \lesssim \psi'$. Therefore, the last integral can also be bounded by $\lesssim \int_\R \psi' |v_1|^3 \dd y$.  
Next, we have
\[\begin{split}
 J_2 &:= \int_\R \left(\psi \partial_y v_1 + \frac 1 2 \psi' v_1\right) \cos U \sin v_1 \dd y\\
 &= \int_\R \left( \psi \cos U \partial_y (1-\cos v_1) + \frac 1 2 \psi ' v_1 \cos U \sin v_1 \right) \dd y\\
 &= \int_\R \psi' \cos U \left( \cos (v_1)-1 + \frac 1 2 v_1 \sin v_1\right) \dd y + \int_\R \psi U' \sin U (1-\cos v_1) \dd y.
\end{split} \]
Standard Taylor expansions show $|1-\cos v_1| \lesssim v_1^2$ and $|\cos(v_1) - 1 + \frac 1 2 v_1 \sin v_1| \lesssim v_1^4$. Therefore, with $|\sin U(y)| + |U'(y)| \lesssim \delta e^{-(1-\delta)|y|}$,
\[ |J_2| \lesssim \int_\R \psi' v_1^4 \dd y + \delta^2 \int_\R e^{-2(1-\delta)|y|} v_1^2 \dd y \leq \int_\R \psi' v_1^4 \dd y + \delta^2 \|v_1 \|_{L^2_\omega}^2.\]
Applying Lemma \ref{l:psi-prime} again, the proof is complete.
\end{proof}

Now we specialize to the case $U\equiv 0$, and consider general perturbations (without parity assumptions). 
\begin{proposition}\label{p:integral-bound}
Let $F$ be as in Theorem \ref{t:asymptotic}. For any global solution $u(t,y)$ of \eqref{e:main-y} with 
\[ \|u(t)\|_{H^1(\R)}^2 + \|\partial_t u(t)\|_{L^2(\R)}^2 \lesssim \eps^2, \quad t\in [0,\infty)\]
there holds 
\[ \int_0^\infty (\|u(t) \|_{H^1_\omega}^2 + \|\partial_t u(t)\|_{L^2_\omega}^2) \dd t \lesssim \eps^2.\]
\end{proposition}
\begin{proof}
In the above calculations, we choose $\lambda>0$ so that condition \eqref{e:b-c-coercive} in the statement of Theorem \ref{t:asymptotic} is satisfied. From \eqref{e:ddtI}, Lemma \ref{l:B}, and Lemma \ref{l:Fterm}, we have, with $v_1 = u$ and $w = \zeta v_1$,
\begin{equation}\label{e:stability1}
\begin{split}
 - \frac d {dt} \mathcal I(v) &\geq 
 \int_\R \left( (\partial_y w)^2 - \frac 1 {2\lambda^2}  \sech^2(y/\lambda) w^2\right) \dd y 
 - \eps\|\partial_y w\|_{L^2}^2 - \delta \|v_1\|_{H^1_\omega}^2\\
 & \quad - \int_\R \left(\psi b (\partial_y v_1)^2 + \frac 1 2 \psi' c v_1^2 \right)\dd y - \int_\R \left( \frac 1 2 \psi' b + \psi c\right) (\partial_y v_1) v_1 \dd y.
\end{split}
\end{equation}
With $\eps< \frac 1 4$, we use $w = \zeta v_1$ and Young's inequality to write 
\[(1-\eps)(\partial_y w)^2 \geq \frac 3 4 \zeta^2 (\partial_y v_1)^2 + \frac 3 4 (\zeta')^2 v_1^2 + \frac 3 2 \zeta \zeta' v_1 \partial_y v_1    \geq
\frac 1 2 \zeta^2 (\partial_y v_1)^2 - \frac 3 2 (\zeta')^2 v_1^2.\]  
This implies
\[ \begin{split}
\int_\R &\left( (\partial_y w)^2 - \frac 1 {2\lambda^2}  \sech^2(y/\lambda) w^2\right) \dd y  - \eps\|\partial_y w\|_{L^2}^2\\ 
&\geq \int_\R  \left( \frac 1 2\sech^2(y/\lambda) (\partial_yv_1)^2 - \left(\frac 3 2\sech^2(y/\lambda) \tanh^2(y/\lambda) + \frac 1 {2\lambda^2} \sech^4(y/\lambda)\right)v_1^2\right) \dd y\\
&\geq  \int_\R\left( \frac 1 2\sech^2(y/\lambda) (\partial_yv_1)^2 - \frac 3 2\sech^2(y/\lambda) v_1^2\right)\dd y,
\end{split}\]
using $(3/2)\tanh^2(y/\lambda) + 1/(2\lambda^2) \sech^2(y/\lambda) = 3/2 + (1/(2\lambda^2) - 3/2)\sech^2(y/\lambda) \leq 3/2$, since $\lambda^2 > 3$. For the terms with $b$ and $c$, integrating by parts, we have
\[ \begin{split}
- \int_\R \left(\psi b (\partial_y v_1)^2 + \frac 1 2 \psi' c v_1^2 \right)\dd y &-  \int_\R \left( \frac 1 2 \psi' b + \psi c\right) (\partial_y v_1) v_1 \dd y\\
&= -\int_\R \left( \psi b (\partial_y v_1)^2 - \frac 1 4 (2 \psi c'  + (\psi' b)') v_1^2 \right) \dd y.
\end{split}
\]
Using our assumption \eqref{e:b-c-coercive} on $b$ and $c$, we see that $\frac 1 4 (2\psi c' + (\psi' b)') \geq  2 \sech^2(y/\lambda)$, and $-\psi b \geq -\frac 1 4 \sech^2(y/\lambda)$.  
The estimate \eqref{e:stability1} now becomes
\begin{equation}\label{e:Ibound}
 - \frac d {dt} \mathcal I(v) \geq \int_\R \left( \frac 1 4 \sech^2(y/\lambda) (\partial_y v_1)^2 + \frac 1 2 \sech^2(y/\lambda) v_1^2 \right) \dd y - \delta \|v_1\|_{H^1_\omega}^2 \geq C_1 \|v_1\|_{H^1_\omega},
\end{equation}
since $\sech^2(y/\lambda) \gtrsim \sech(y)$.

Next, we notice that
\begin{equation}\label{e:sech-bound}
\begin{split}
 \frac d {dt} \int_\R \sech(y) v_1 v_2 \dd y &= \int_\R \sech(y) (v_2^2 + v_1 \partial_t v_2)\dd y \\
 &= \int_\R \sech(y) [v_2^2 + v_1(\partial_y^2 v_1 + b \partial_y v_1 + cv_1 + \mathcal N(U,v_1))]\dd y\\
 &= \int_\R \left[\sech(y)(v_2^2 - (\partial_y v_1)^2 + v_1 \mathcal N(U,v_1)) - \sech'(y) v_1 \partial_y v_1\right]\dd y \\
 &\quad +\int_\R \sech(y)\left[ b v_1 \partial_y v_1 + c v_1^2\right] \dd y\\
 &= \int_\R \left[\sech(y)(v_2^2 - (\partial_y v_1)^2 + v_1 \mathcal N(U,v_1)) + \frac 1 2 \sech''(y) v_1^2 \right]\dd y \\
 &\quad -\frac 1 2 \int_\R (\sech'(y)b + \sech(y) b' + 2\sech(y)c)v_1^2 \dd y\\ 
 &\geq \|v_2\|_{L^2_\omega}^2 - \|v_1\|_{H^1_\omega}^2 +\int_\R \left(\sech(y) v_1 \mathcal N(U,v_1) + \frac 1 2 \sech''(y) v_1^2 \right)\dd y \\
 &\quad - C(\|b\|_{W^{1,\infty}}+ \|c\|_{L^\infty}) \|v_1\|_{H^1_\omega}^2,
 \end{split}
 \end{equation} 
using $|\sech'(y)| = |\sech(y)\tanh(y)|\lesssim \sech(y)$. To bound the last expression from below, note that $|\sech''(y)| = \sech(y)|\tanh^2(y)-\sech^2(y)| \lesssim \sech(y)$. Also, since $\mathcal N(U,v_1) = F'(0) - F'(v_1) = F''(z) v_1$ for some $|z|\leq |v_1|$, we have $|v_1\mathcal N(U,v_1)|\lesssim v_1^2$. We then have
 \begin{equation}\label{e:sech}
  \frac d {dt} \int_\R \sech(y) v_1 v_2 \dd y \geq \|v_2\|_{L^2_\omega}^2 - C\|v_1\|_{H^1_\omega}^2,
  \end{equation}
with $C>0$ depending on $b$ and $c$. Combining this with \eqref{e:Ibound}, we conclude in the following way: integrate \eqref{e:Ibound} from $0$ to $t_0$ and use the assumption that $\|v_1\|_{H^1} + \|v_2\|_{L^2} \lesssim \eps$ to bound $\mathcal I$ from above, since $\mathcal I(v) \lesssim \|v_1\|_{H^1}\|v_2\|_{L^2} \lesssim \eps^2$:
 \[ \int_0^{t_0} \|v_1\|_{H^1_\omega}^2 \leq -\mathcal I(v)(t_0) + \mathcal I(v)(0) \lesssim \eps^2. \]
 Sending $t_0 \to \infty$, we have
 \[ \int_0^\infty \|v_1\|_{H^1_\omega}^2 \lesssim \eps^2.\]
 Next, we similarly integrate \eqref{e:sech} from $0$ to $t_0$, use the bound $\|v_1\|_{H^1} + \|v_2\|_{L^2} \lesssim \eps$ to write $\int_\R \sech(y) v_1 v_2 \dd y \lesssim \|v_1\|_{L^2}\|v_2\|_{L^2} \lesssim \eps^2$, and send $t_0\to \infty$:
 \[ \int_0^\infty \|v_2\|_{L^2_\omega}^2 \dd t \lesssim \int_0^\infty\|v_1\|_{H^1_\omega}^2 \dd t + \eps^2 \lesssim \eps^2,\]
as claimed.
\end{proof}

The following result, in combination with orbital stability, implies Theorem \ref{t:asymptotic}. On its own, Theorem \ref{t:no-breathers} is a no-breathers theorem ruling out the existence small global solutions that do not vanish as $t\to\infty$.

\begin{theorem}\label{t:no-breathers}
With $F$ as in Theorem \ref{t:asymptotic}, for any global solution $u(t,y)$ of \eqref{e:main-y} with 
\[ \|u(t)\|_{H^1(\R)}^2 + \|\partial_t u(t)\|_{L^2(\R)}^2 \lesssim \eps^2, \quad t\in [0,\infty)\]
there holds 
\[ \lim_{t\to \infty} \left(\|u(t)\|_{H^1(I)} + \|\partial_t u(t)\|_{L^2(I)}\right) = 0,\]
for any bounded interval $I\subset \R$.
\end{theorem}
\begin{proof}
Following the proof of \cite[Theorem 1.1]{KMM2017breathers}, we consider
\[ \|v(t)\|_{H^1_\omega\times L^2_\omega}^2 = \int_\R \sech(y) [(\partial_y v_1)^2 + v_1^2 + v_2^2]\dd y.\]
Differentiating and using the equation \eqref{e:vector} satisfied by $(v_1,v_2)$, we have
\[\begin{split}
\frac d {dt} \|v(t)\|_{H^1_\omega\times L^2_\omega}^2  &= 2 \int_\R \sech(y) [\partial_y v_1 \partial_t \partial_y v_1 + v_1 \partial_t v_1 + v_2 \partial_t v_2] \dd y\\
&= 2\int_\R \sech(y) [ \partial_y v_1 \partial_y v_2 + v_1 v_2 + v_2(\partial_y^2 v_1 + b\partial_y v_1 + cv_1 +\mathcal N(U,v_1))] \dd y\\
&= 2 \int_\R \sech(y) v_2 ((1+c)v_1 + b\partial_y v_1 + \mathcal N(U,v_1)) \dd y + \int_\R \sech(y) \partial_y (v_2 \partial_y v_1) \dd y \\
&= 2 \int_\R \sech(y) v_2 ((1+c)v_1 + b\partial_y v_1 + \mathcal N(U,v_1)) \dd y - \int_\R \sech'(y) v_2 \partial_y v_1 \dd y.
\end{split}\]
With $b,c\in L^\infty(\R)$, applying Young's inequality and $|\mathcal N(U,v_1)| \lesssim v_1$, we have
\[ \left|\frac d {dt} \|v(t)\|_{H^1_\omega\times L^2_\omega}^2 \right| \lesssim \int_\R \sech(y)[v_1^2 + v_2^2 + (\partial_y v_1)^2] \dd y \leq \|v_1\|_{H^1_\omega}^2 + \|v_2\|_{L^2_\omega}^2.\]
By Proposition \ref{p:integral-bound}, there is a sequence $t_n\to \infty$ such that $\|v(t_n)\|_{H^1_\omega\times L^2_\omega}^2  \to 0$. For any $t\geq 0$, we can integrate our bound for $|\frac d {dt} \|v(t)\|_{H^1_\omega\times L^2_\omega}^2 |$ from $t$ to $t_n$ to obtain
\[| \|v(t_n)\|_{H^1_\omega\times L^2_\omega}^2  -\|v(t)\|_{H^1_\omega\times L^2_\omega}^2 |  \lesssim \int_{t}^{t_n} (\|v_1(t')\|_{H^1_\omega}^2 + \|v_2(t')\|_{L^2_\omega}^2) \dd t'.  \]
Sending $n\to \infty$, we obtain
\[   \|v(t)\|_{H^1_\omega\times L^2_\omega}^2   \lesssim \int_{t}^{\infty} (\|v_1(t')\|_{H^1_\omega}^2 + \|v_2(t')\|_{L^2_\omega}^2) \dd t'.  \]
 Since $(\|v_1(t')\|_{H^1_\omega}^2 + \|v_2(t')\|_{L^2_\omega}^2)$ is integrable for large times thanks to Proposition \ref{p:integral-bound}, we must have $\|v(t)\|_{H^1_\omega\times L^2_\omega}^2  \to 0$ as $t\to \infty$, which implies the statement of the theorem.
\end{proof}
Combining Theorem \ref{t:no-odd-breathers} with the orbital stability result of Proposition \ref{p:orbital-constant} implies Theorem \ref{t:asymptotic}.

Next, we switch to the setting of Theorem \ref{t:asymptotic-odd}, i.e. $\xi \neq 0$, $F(u) = 1-\cos u$, and $U$ the stationary solution guaranteed by Theorem \ref{t:near-constant}. We assume the initial data are odd in $y$, $b$ is odd, and $c$ is even.

\begin{proposition}\label{p:integral-bound-odd}
Let $F$, $b$, and $c$ be as in Theorem \ref{t:asymptotic-odd}. For any odd (in $y$) global solution $u(t,y)$ of \eqref{e:main-y} with 
\[ \|u(t)-U\|_{H^1(\R)}^2 + \|\partial_t u(t)\|_{L^2(\R)}^2 \lesssim \eps^2, \quad t\in [0,\infty)\]
there holds 
\[ \int_0^\infty (\|u(t) - U\|_{H^1_\omega}^2 + \|\partial_t u(t)\|_{L^2_\omega}^2) \dd t \lesssim \eps^2.\]
\end{proposition}
\begin{proof}
The proof is very similar to Proposition \ref{p:integral-bound}. We choose the value $\lambda = 100$, which is not essential, but allows us to borrow some calcuations from \cite{KMM2017breathers}. Instead of \eqref{e:stability1}, we have, using \eqref{e:odd-coercivity} (which, as mentioned above, comes from \cite[Lemma 2.1]{KMM2017breathers}),
\begin{equation}\label{e:stability2}
\begin{split}
 - \frac d {dt} \mathcal I(v) &\geq 
 \frac 3 4 \int_\R  (\partial_y w)^2 \dd y 
 - \eps\|\partial_y w\|_{L^2}^2 - \delta \|v_1\|_{H^1_\omega}^2\\
 & \quad - \int_\R \left(\psi b (\partial_y v_1)^2 + \frac 1 2 \psi' c v_1^2 \right)\dd y - \int_\R \left( \frac 1 2 \psi' b + \psi c\right) (\partial_y v_1) v_1 \dd y.
\end{split}
\end{equation}
We also quote from \cite[Formula (2.20)]{KMM2017breathers} the estimate
\[  \|\partial_y w\|_{L^2} \gtrsim \|v_1\|_{H^1_\omega},  \]
which holds for $\lambda = 100$. With $\eps$ and $\delta$ small enough, we have
\[ -\frac d {dt} \mathcal I(v) \geq  C \|v_1\|_{H^1_\omega}^2    -\int_\R \left( \psi b (\partial_y v_1)^2 - \frac 1 4 (2 \psi c'  + (\psi' b)') v_1^2 \right) \dd y\geq \frac 1 2 \|v_1\|_{H^1_\omega},\]
by the sign conditions \eqref{e:sign-condition} on $b$ and $c$. The remainder of the proof proceeds in the same manner as Proposition \ref{p:integral-bound}.
\end{proof}

We now obtain a corresponding result ruling out small, odd breathers. The proof is essentially identical to Theorem \ref{t:no-breathers}.

\begin{theorem}\label{t:no-odd-breathers}
With $F$ as in Theorem \ref{t:asymptotic}, for any odd (in $y$) global solution $u(t,y)$ of \eqref{e:main-y} with 
\[ \|u(t,\cdot)-U\|_{H^1(\R)}^2 + \|\partial_t u(t,\cdot)\|_{L^2(\R)}^2 \lesssim \eps^2, \quad t\in [0,\infty)\]
there holds 
\[ \lim_{t\to \infty} \left(\|u(t)-U\|_{H^1(I)} + \|\partial_t u(t)\|_{L^2(I)}\right) = 0,\]
for any bounded interval $I\subset \R$.
\end{theorem}

Theorem \ref{t:asymptotic-odd} now follows from Proposition \ref{p:orbital-constant} and Theorem \ref{t:no-odd-breathers}.

\bibliographystyle{abbrv}
\bibliography{sineGordon}

\begin{thebibliography}{10}

\bibitem{alammari2021}
M.~Alammari and S.~Snelson.
\newblock Linear and orbital stability analysis for solitary-wave solutions of
  variable-coefficient scalar-field equations.
\newblock {\em J. Hyperbolic Differ. Equ.}, 19(1):175--201, 2022.

\bibitem{AMP2020sinegordon}
M.~A. Alejo, C.~Muñoz, and J.~M. Palacios.
\newblock On {Asymptotic} {Stability} of the {Sine}-{Gordon} {Kink} in the
  {Energy} {Space}.
\newblock {\em Communications in Mathematical Physics}, June 2023.

\bibitem{bambusi2011}
D.~Bambusi and S.~Cuccagna.
\newblock On dispersion of small energy solutions to the nonlinear {K}lein
  {G}ordon equation with a potential.
\newblock {\em Amer. J. Math.}, 133(5):1421--1468, 2011.

\bibitem{chen2020sine-gordon}
G.~Chen, J.~Liu, and B.~Lu.
\newblock Long-time asymptotics and stability for the sine-{G}ordon equation.
\newblock {\em Preprint. arXiv:2009.04260}, 2020.

\bibitem{cuccagna2008kink}
S.~Cuccagna.
\newblock On asymptotic stability in 3{D} of kinks for the {$\phi^4$} model.
\newblock {\em Trans. Amer. Math. Soc.}, 360(5):2581--2614, 2008.

\bibitem{cuccagna2014nls}
S.~Cuccagna, V.~Georgiev, and N.~Visciglia.
\newblock Decay and scattering of small solutions of pure power {NLS} in {$\Bbb
  R$} with {$p > 3$} and with a potential.
\newblock {\em Comm. Pure Appl. Math.}, 67(6):957--981, 2014.

\bibitem{cuccagna2019nls}
S.~Cuccagna and M.~Maeda.
\newblock On stability of small solitons of the 1-{D} {NLS} with a trapping
  delta potential.
\newblock {\em SIAM J. Math. Anal.}, 51(6):4311--4331, 2019.

\bibitem{cuccagna2021nls}
S.~Cuccagna and M.~Maeda.
\newblock On selection of standing wave at small energy in the 1{D} cubic
  {S}chr\"{o}dinger equation with a trapping potential.
\newblock {\em Comm. Math. Phys.}, 396(3):1135--1186, 2022.

\bibitem{cuccagna2022kink}
S.~Cuccagna and M.~Maeda.
\newblock Asymptotic stability of kink with internal modes under odd
  perturbation.
\newblock {\em NoDEA Nonlinear Differential Equations Appl.}, 30(1):Paper No.
  1, 47, 2023.

\bibitem{cuenda2011sg}
S.~Cuenda, N.~R. Quintero, and A.~S\'{a}nchez.
\newblock Sine-{G}ordon wobbles through {B}\"{a}cklund transformations.
\newblock {\em Discrete Contin. Dyn. Syst. Ser. S}, 4(5):1047--1056, 2011.

\bibitem{sine-gordon-book}
J.~Cuevas-Maraver, P.~G. Kevrekidis, and F.~Williams, editors.
\newblock {\em The sine-{G}ordon model and its applications}, volume~10 of {\em
  Nonlinear Systems and Complexity}.
\newblock Springer, Cham, 2014.

\bibitem{delort2001nlkg}
J.-M. Delort.
\newblock Existence globale et comportement asymptotique pour l'\'{e}quation de
  {K}lein-{G}ordon quasi lin\'{e}aire \`a donn\'{e}es petites en dimension 1.
\newblock {\em Ann. Sci. \'{E}cole Norm. Sup. (4)}, 34(1):1--61, 2001.

\bibitem{delort2016nls}
J.-M. Delort.
\newblock Modified scattering for odd solutions of cubic nonlinear
  {S}chr\"odinger equations with potential in dimension one.
\newblock {\em Preprint: hal-01396705}, 2016.

\bibitem{denzler1993sinegordon}
J.~Denzler.
\newblock Nonpersistence of breather families for the perturbed sine {G}ordon
  equation.
\newblock {\em Comm. Math. Phys.}, 158(2):397--430, 1993.

\bibitem{germain2020nlkg}
P.~Germain and F.~Pusateri.
\newblock Quadratic {K}lein-{G}ordon equations with a potential in one
  dimension.
\newblock {\em Forum Math. Pi}, 10:Paper No. e17, 2022.

\bibitem{germain2018nls-potential}
P.~Germain, F.~Pusateri, and F.~Rousset.
\newblock The nonlinear {S}chr\"{o}dinger equation with a potential.
\newblock {\em Ann. Inst. H. Poincar\'{e} Anal. Non Lin\'{e}aire},
  35(6):1477--1530, 2018.

\bibitem{hgm2018cc}
B.~Harrop-Griffiths and J.~L. Marzuola.
\newblock Small data global solutions for the {C}amassa-{C}hoi equations.
\newblock {\em Nonlinearity}, 31(5):1868--1904, 2018.

\bibitem{hayashi1998nls}
N.~Hayashi and P.~I. Naumkin.
\newblock Asymptotic behavior in time of solutions to the derivative nonlinear
  {S}chr\"{o}dinger equation.
\newblock {\em Ann. Inst. H. Poincar\'{e} Phys. Th\'{e}or.}, 68(2):159--177,
  1998.

\bibitem{ifrim2015nls}
M.~Ifrim and D.~Tataru.
\newblock Global bounds for the cubic nonlinear {S}chr\"{o}dinger equation
  ({NLS}) in one space dimension.
\newblock {\em Nonlinearity}, 28(8):2661--2675, 2015.

\bibitem{ifrim2019b-o}
M.~Ifrim and D.~Tataru.
\newblock Well-posedness and dispersive decay of small data solutions for the
  {B}enjamin-{O}no equation.
\newblock {\em Ann. Sci. \'{E}c. Norm. Sup\'{e}r. (4)}, 52(2):297--335, 2019.

\bibitem{ivancevic-sinegordon}
V.~G. Ivancevic and T.~T. Ivancevic.
\newblock Sine-{G}ordon solitons, kinks and breathers as physical models of
  nonlinear excitations in living cellular structures.
\newblock {\em J. Geom. Symmetry Phys.}, 31:1--56, 2013.

\bibitem{khare}
A.~Khare, I.~C. Christov, and A.~Saxena.
\newblock Successive phase transitions and kink solutions in
  ${\ensuremath{\phi}}^{8}$, ${\ensuremath{\phi}}^{10}$, and
  ${\ensuremath{\phi}}^{12}$ field theories.
\newblock {\em Phys. Rev. E}, 90:023208, Aug 2014.

\bibitem{kohler2021}
S.~Kohler and W.~Reichel.
\newblock Breather solutions for a quasilinear $(1+1)$-dimensional wave
  equation.
\newblock {\em Studies in Applied Mathematics}, 148:689--714, 2022.

\bibitem{KMkink}
M.~Kowalczyk and Y.~Martel.
\newblock Kink dynamics under odd perturbations for $(1+1)$-scalar field models
  with one internal mode.
\newblock {\em Preprint. ArXiv:2203.04143}.

\bibitem{KMMphi4}
M.~Kowalczyk, Y.~Martel, and C.~Mu\~{n}oz.
\newblock Kink dynamics in the {$\phi^4$} model: asymptotic stability for odd
  perturbations in the energy space.
\newblock {\em J. Amer. Math. Soc.}, 30(3):769--798, 2017.

\bibitem{KMM2017breathers}
M.~Kowalczyk, Y.~Martel, and C.~Mu\~{n}oz.
\newblock Nonexistence of small, odd breathers for a class of nonlinear wave
  equations.
\newblock {\em Lett. Math. Phys.}, 107(5):921--931, 2017.

\bibitem{KMMV2020kink}
M.~Kowalczyk, Y.~Martel, C.~Mu\~{n}oz, and H.~Van Den~Bosch.
\newblock A {S}ufficient {C}ondition for {A}symptotic {S}tability of {K}inks in
  {G}eneral (1+1)-{S}calar {F}ield {M}odels.
\newblock {\em Ann. PDE}, 7(1):10, 2021.

\bibitem{segur1987}
M.~D. Kruskal and H.~Segur.
\newblock Nonexistence of small-amplitude breather solutions in {$\phi^4$}
  theory.
\newblock {\em Phys. Rev. Lett.}, 58(8):747--750, 1987.

\bibitem{li2022quadratic}
Y.~Li and J.~Lührmann.
\newblock Soliton dynamics for the 1d quadratic {K}lein-{G}ordon equation with
  symmetry.
\newblock {\em Journal of Differential Equations}, 344:172--202, 2023.

\bibitem{lindblad2019scattering}
H.~Lindblad, J.~L\"{u}hrmann, and A.~Soffer.
\newblock Decay and asymptotics for the one-dimensional {K}lein-{G}ordon
  equation with variable coefficient cubic nonlinearities.
\newblock {\em SIAM J. Math. Anal.}, 52(6):6379--6411, 2020.

\bibitem{lindblad2015scattering}
H.~Lindblad and A.~Soffer.
\newblock Scattering for the {K}lein-{G}ordon equation with quadratic and
  variable coefficient cubic nonlinearities.
\newblock {\em Trans. Amer. Math. Soc.}, 367(12):8861--8909, 2015.

\bibitem{lindblad-tao}
H.~Lindblad and T.~Tao.
\newblock Asymptotic decay for a one-dimensional nonlinear wave equation.
\newblock {\em Anal. PDE}, 5(2):411--422, 2012.

\bibitem{lohe}
M.~A. Lohe.
\newblock Soliton structures in {$P{(\ensuremath{\varphi})}_{2}$}.
\newblock {\em Phys. Rev. D}, 20:3120--3130, 1979.

\bibitem{top-sol}
N.~Manton and P.~Sutcliffe.
\newblock {\em Topological solitons}.
\newblock Cambridge Monographs on Mathematical Physics. Cambridge University
  Press, Cambridge, 2004.

\bibitem{martinez2020hartree}
M.~E. Mart\'{\i}nez.
\newblock Decay of small odd solutions for long range {S}chr\"{o}dinger and
  {H}artree equations in one dimension.
\newblock {\em Nonlinearity}, 33(3):1156--1182, 2020.

\bibitem{naumkin2016nls}
I.~P. Naumkin.
\newblock Sharp asymptotic behavior of solutions for cubic nonlinear
  {S}chr\"{o}dinger equations with a potential.
\newblock {\em J. Math. Phys.}, 57(5):051501, 31, 2016.

\bibitem{ross2019}
R.~M. Ross, P.~G. Kevrekidis, D.~K. Campbell, R.~Decker, and A.~Demirkaya.
\newblock {$\phi^4$} solitary waves in a parabolic potential: existence,
  stability, and collisional dynamics.
\newblock In {\em A dynamical perspective on the {$\phi ^4$} model}, volume~26
  of {\em Nonlinear Syst. Complex.}, pages 213--234. Springer, Cham, 2019.

\bibitem{snelson2016stability}
S.~Snelson.
\newblock Asymptotic stability for odd perturbations of the stationary kink in
  the variable-speed $\phi^4$ model.
\newblock {\em Trans, Amer. Math. Soc.}, 370(10):7437--7460, 2018.

\bibitem{soffer1999resonance}
A.~Soffer and M.~I. Weinstein.
\newblock Resonances, radiation damping and instability in {H}amiltonian
  nonlinear wave equations.
\newblock {\em Invent. Math.}, 136(1):9--74, 1999.

\bibitem{sterbenz2016decay}
J.~Sterbenz.
\newblock Dispersive decay for the 1{D} {K}lein-{G}ordon equation with variable
  coefficient nonlinearities.
\newblock {\em Trans. Amer. Math. Soc.}, 368(3):2081--2113, 2016.

\bibitem{kinks-domainwalls}
T.~Vachaspati.
\newblock {\em Kinks and domain walls}.
\newblock Cambridge University Press, New York, 2006.
\newblock An introduction to classical and quantum solitons.

\bibitem{vuillermot}
P.-A. Vuillermot.
\newblock Nonexistence of spatially localized free vibrations for a class of
  nonlinear wave equations.
\newblock {\em Comment. Math. Helv.}, 62(4):573--586, 1987.

\end{thebibliography}

\end{document}